\documentclass[11pt,reqno]{amsart}

\usepackage{enumerate,tikz-cd,mathtools,amssymb}

\usepackage[a4paper,top=3cm,bottom=2.5cm,left=3cm,right=3cm,marginparwidth=1.75cm]{geometry}

\usepackage{verbatim}

\usepackage{array, makecell}
\usepackage{float}
\usepackage{longtable}
\usepackage{graphicx}
\usepackage{caption}
\usepackage{libertine}
\usepackage{xcolor}
\usepackage{courier}
\usepackage{listings}
\allowdisplaybreaks
\usepackage{microtype}

\definecolor{commentgreen}{RGB}{2,112,10}
\definecolor{eminence}{RGB}{108,48,130}
\definecolor{frenchplum}{RGB}{149,20,83}

\definecolor{ffqqqq}{RGB}{215,25,28}
\definecolor{cccccc}{RGB}{171,217,233}

\lstdefinelanguage{Sage}
{
keywords={load, matrix, from, import},
emph={},
}

\usepackage[pdftex,
              pdfauthor={Riccardo Moschetti, Franco Rota, Luca Schaffler},
              pdftitle={A computational view on the non-degeneracy invariant for Enriques surfaces},
              pdfsubject={non-degeneracy invariant, Enriques surfaces, computational view},
              pdfkeywords={Enriques surface, elliptic fibration, rational curve, non-degeneracy invariant, Fano polarization, Moschetti,  Rota, Schaffler},
              colorlinks=true
              ]{hyperref}

\usepackage[all,cmtip]{xy}

 \usetikzlibrary{shapes}
\tikzset{
    partial ellipse/.style args={#1:#2:#3}{
        insert path={+ (#1:#3) arc (#1:#2:#3)}
    }
}

\makeatletter
\newtheorem*{rep@theorem}{\rep@title}
\newcommand{\newreptheorem}[2]{
\newenvironment{rep#1}[1]{
 \def\rep@title{#2 \ref{##1}}
 \begin{rep@theorem}}
 {\end{rep@theorem}}}
\makeatother

\theoremstyle{plain}
\newtheorem{thm}{Theorem}[section]

\newtheorem{lemma}[thm]{Lemma}

\newtheorem{proposition}[thm]{Proposition}

\newtheorem*{thm*}{Theorem}
\newtheorem*{corollary*}{Corollary}
\newtheorem*{lemma*}{Lemma}
\newtheorem*{ld*}{Lemma/Definition}
\newtheorem*{proposition*}{Proposition}
\newtheorem*{assumption*}{Assumption}

\theoremstyle{definition}
\newtheorem{definition}[thm]{Definition}
\newtheorem{example}[thm]{Example}

\newtheorem*{definition*}{Definition}
\newtheorem*{example*}{Example}
\newtheorem*{xca*}{Exercise}
\newtheorem*{claim*}{Claim}
\newtheorem*{fact*}{Fact}
\newtheorem*{notation*}{Notation}
\newtheorem*{construction*}{Construction}
\newtheorem*{ack*}{Acknowledgements}
\newtheorem*{question*}{Question}
\newtheorem*{problem*}{Problem}
\newtheorem*{conjecture*}{Conjecture}

\theoremstyle{remark}
\newtheorem{remark}[thm]{Remark}

\DeclareMathOperator{\nd}{nd}
\DeclareMathOperator{\cnd}{cnd}
\DeclareMathOperator{\Num}{Num}
\newcommand{\sR}{\mathcal{R}}
\newcommand{\sO}{\mathcal{O}}
\newcommand{\pr}[1]{\mathbb P^{#1}}
\newcommand{\Ku}{\mathcal{K}u}
\newcommand{\sE}{\mathcal{E}}
\DeclareMathOperator{\Bl}{Bl}
\newcommand{\Z}{\mathbb{Z}}
\DeclareMathOperator{\Coh}{Coh}
\newcommand{\sL}{\mathcal{L}}
\DeclareMathOperator{\SL}{SL}

\newcommand{\f}{\mathrm{F}}
\newcommand{\hf}{\mathrm{HF}}

\begin{document}
\title[A computational view on the non-degeneracy invariant for Enriques surfaces]{A computational view on the non-degeneracy invariant\\for Enriques surfaces}

\author[R. Moschetti]{Riccardo Moschetti}
\address{RM: Department of Mathematics G. Peano, University of Turin, via Carlo Alberto 10, 10123 Torino, Italy} 
\email{riccardo.moschetti@unito.it}

\author[F. Rota]{Franco Rota}
\address{FR: School of Mathematics and Statistics, University of Glasgow, Glasgow G12 8QQ, United Kingdom} 
\email{franco.rota@glasgow.ac.uk}

\author[L. Schaffler]{Luca Schaffler}
\address{LS: Dipartimento di Matematica e Fisica, Universit\`a degli Studi Roma Tre, Largo San Leonardo Murialdo 1, 00146, Roma, Italy}
\email{luca.schaffler@uniroma3.it}


\begin{abstract}
For an Enriques surface $S$, the non-degeneracy invariant $\mathrm{nd}(S)$ retains information on the elliptic fibrations of $S$ and its polarizations. In the current paper, we introduce a combinatorial version of the non-degeneracy invariant which depends on $S$ together with a configuration of smooth rational curves, and gives a lower bound for $\mathrm{nd}(S)$. We provide a SageMath code that computes this combinatorial invariant and we apply it in several examples. First we identify a new family of nodal Enriques surfaces satisfying $\mathrm{nd}(S)=10$ which are not general and with infinite automorphism group. We obtain lower bounds on $\mathrm{nd}(S)$ for the Enriques surfaces with eight disjoint smooth rational curves studied by Mendes Lopes--Pardini. Finally, we recover Dolgachev and Kond\=o's computation of the non-degeneracy invariant of the Enriques surfaces with finite automorphism group and provide additional information on the geometry of their elliptic fibrations.
\end{abstract}

\maketitle

\section{Introduction}
 
%---------------------------------------------------------------------------------

For an Enriques surface $S$, the \emph{non-degeneracy invariant} $\nd(S)$ was introduced in \cite{CD89}. It can be defined as follows.
Enriques surfaces always have an elliptic pencil, and each elliptic pencil has exactly two non-reduced fibers of multiplicity $2$. These fibers, taken with their reduced structure, are called \textit{half-fibers}. Then, $\nd(S)$ is defined to be the maximum number of half-fibers $F_1,\ldots,F_m$ such that 
\begin{equation}
\label{eq:FiFj=1-dij}
F_i\cdot F_j=1-\delta_{ij}
\end{equation}
(note that $F_i^2=0$ automatically for all $i$). We work over an algebraically closed field of characteristic different from $2$ (see Remark~\ref{rmk:programCharacteristic2} for characteristic $2$). It is known that $\nd(S)\leq10$ because $\Num(S)$, the group of divisors on $S$ modulo numerical equivalence, has rank $10$. The inequality $3\leq\nd(S)$ is a theorem of Cossec \cite[Theorem~3]{Cos85}, which was recently re-proven in \cite{MMV22a} and improved to $4\leq\nd(S)$ in \cite{MMV22b}.

If $S$ is an \emph{unnodal} Enriques surface, i.e. $S$ does not contain a smooth rational curve, it is always possible to find such a sequence of length $10$ \cite[Theorem~3.2]{Cos85}. The non-degeneracy invariant for a \emph{general} nodal Enriques surface $S$, which means that the numerical classes of smooth rational curves on $S$ are congruent modulo $2\Num(S)$, is also known to be $10$ (this is a consequence of \cite[\S\,4.2]{DM19} combined with \cite[Lemma~3.2.1]{Cos83}).

For non-general nodal Enriques surfaces the problem of understanding $\nd(S)$ is more subtle. Examples of such Enriques surfaces are the ones with finite automorphism group, which were classified by Kond\=o into seven irreducible families \cite{Kon86}. The non-degeneracy invariants of these surfaces are computed in \cite[\S\,8.9]{DK22} as follows:
\begin{center}
\begin{tabular}{|c|c|c|c|c|c|c|c|}
\hline
Type & I & II & III & IV & V & VI & VII \\
\hline
$\nd$ & 4 & 7 & 8 & 10 & 7 & 10 & 10 \\
\hline
\end{tabular}
\end{center}
Another class of non-general nodal Enriques surfaces, but with infinite automorphism group, is the $4$-dimensional family of Hessian Enriques surfaces. These satisfy $\nd(S)=10$ (see \cite[\S,4.1--\S\,4.3]{Dol18}). At the moment, no Enriques surface with infinite automorphism group is known to satisfy $\nd(S)<10$, and examples of Enriques surfaces with $\nd(S)=5,6,9$ are not known.

\subsection{Main results}

In this work, we outline an approach to studying the non-degeneracy invariant for an Enriques surface $S$. Suppose we have a configuration $\mathcal{R}=\{R_1,\ldots,R_\ell\}$ of smooth rational curves on $S$. Let $C$ be a curve on $S$ which appears in the Kodaira classification of singular fibers of elliptic fibrations, and whose irreducible components are elements of $\mathcal{R}$. By general theory, either $C$ or $\frac 12 C$ is linearly equivalent to a half-fiber. Denote by $\mathsf{HF}(S,\mathcal{R})$ the set of numerical equivalence classes of half-fibers which arise from $\mathcal{R}$ in this way. We can then define the \emph{combinatorial non-degeneracy invariant} $\cnd(S,\mathcal{R})$ as the maximum $m$ such that there exist $f_1,\ldots,f_m\in\mathsf{HF}(S,\mathcal{R})$ satisfying \eqref{eq:FiFj=1-dij}. Since it only considers half-fibers supported on $\mathcal{R}$, $\cnd(S,\mathcal{R})$ gives a lower bound for $\nd(S)$, and has the advantage that its computation can be implemented with a computer.

In this direction, our main contribution is the creation of a piece of code, available at \cite{MRS22} and written in SageMath  \cite{Sag22}, which computes $\cnd(S,\mathcal{R})$ given a configuration of smooth rational curves $\mathcal{R}$ on an Enriques surface $S$. The input of the algorithm is the intersection matrix of the curves in $\mathcal{R}$ together with a basis for $\Num(S)$. The latter is used to determine if a given elliptic configuration from $\mathcal{R}$ is a fiber or a half-fiber in $S$. Afterwards, the code recursively checks all the possible sequences of half-fibers and obtains $\cnd(S,\mathcal{R})$. A by-product of the computation is also a list of all the sequences of elements in $\mathsf{HF}(S,\mathcal{R})$ satisfying \eqref{eq:FiFj=1-dij} and which cannot be further extended (we call such sequences \emph{saturated}, see \S\,\ref{ssec:SaturatedIsoSequences}).

Then, we apply our computer code to several examples of interest (for simplicity, over $\mathbb{C}$): 
\begin{enumerate}
\item  In \S\,\ref{sec:D16} we consider the $4$-dimensional family of $D_{1,6}$-polarized Enriques surfaces: these arise as the minimal resolution of an appropriate $\mathbb{Z}_2^2$-cover of $\mathbb{P}^2$ branched along six general lines. We show that $\nd(S)=10$: this constitutes a new example because these Enriques surfaces are not general nodal, have infinite automorphism group, and they are not of Hessian type (see Remark~\ref{D16-new-example});
\item There are two families of Enriques surfaces with eight disjoint smooth rational curves \cite{MLP02}. Every such surface $S$ comes with a distinguished set $\sR$ of $12$ smooth rational curves, whose dual graphs are pictured in Figures~\ref{fig_dual_graph_LP1} and \ref{fig_dual_graph_LP2}. In \S\,\ref{sec:LopesPardini} we compute $\cnd(S,\mathcal R)=8$ (resp. $\cnd(S,\mathcal{R})=5$) for the members of the first (resp. second) family.
\item In \S\,\ref{sec:ndKondoExamples} we revisit the Enriques surfaces with finite automorphism group. If $S$ is one of these and $\mathcal{R}$ is the (finite) set of smooth rational curves on $S$, then $\mathsf{HF}(S,\mathcal{R})$ contains all the classes of half-fibers (this follows from the work in \cite{Kon86}), so $\cnd(S,\mathcal{R})=\nd(S)$. In addition to recovering the computation of $\nd(S)$ in \cite{DK22}, we
\begin{itemize}
\item provide explicit sequences of half-fibers realizing $\nd(S)$;
\item list all the saturated sequences;
\item provide alternative views on the dual graphs of smooth rational curves in the Enriques surfaces of type III, IV, V, VI (see Figures~\ref{fig_dual_graph_K3}, \ref{fig_dual_graph_K4}, \ref{fig_dual_graph_K5}, \ref{fig_dual_graph_K6}, respectively), which make the symmetries of the graphs more evident.
\end{itemize}
\end{enumerate}

In the two examples from \cite{MLP02} discussed in \S\,\ref{sec:LopesPardini}, $\cnd(S,\mathcal{R})$ produces a lower bound for $\nd(S)$. In each example, we can use the geometry of the K3 surface covering $S$ to find explicit smooth rational curves on $S$ not in $\mathcal{R}$, and use these to define a new set $\mathcal{R}'\supsetneq\mathcal{R}$. It turns out that our code computes $\cnd(S,\mathcal{R}')=\cnd(S,\mathcal{R})$, and several attempts in this direction make us ask whether $\cnd(S,\mathcal{R})=\nd(S)$. Although we do not elaborate on this aspect in the current paper, we believe that it is worthwhile to understand these examples as a first step towards determining criteria for equality of the invariants, which is an interesting and challenging question. Additionally, it would be interesting to apply \texttt{CndFinder} to other examples of Enriques surfaces with a distinguished configuration of smooth rational curves, such as the one in \cite[Remark~3.9]{FV21}.

%-------------------------------------------------------------------------------

\subsection{Applications}
\label{geom-and-comb-asp-of-nd}

A sequence $\{f_i\}_{i=1}^{10}$ of classes in $\Num(S)$ satisfying \eqref{eq:FiFj=1-dij} encodes rich geometric information about $S$. First of all, the quantity $\frac 13(f_1+\ldots+f_{10})\in \Num(S)$ is the class of a nef divisor $\Delta$ called \emph{Fano polarization}, which defines a map from $S$ to a normal surface of degree $10$ in $\pr 5$ called \emph{Fano model}. We have that $\nd(S)=10$ if and only if $S$ admits a very ample Fano polarization (see the discussion in \cite[\S\,2.3]{DM19}).

The non-degeneracy invariant also plays an important role in the study of the bounded derived category $D^b(\Coh(S))$ of coherent sheaves on $S$, which is known to determine $S$ up to isomorphism \cite{BM01,HLT21}. It turns out that $(f_1,\ldots,f_{10})$ defines a subcategory of $D^b(\Coh(S))$, called \textit{Kuznetsov component}. This subcategory determines $S$ up to isomorphism, as proven in \cite[Theorem~A]{LNSZ21} for $\nd(S)=10$. This was extended to any value of $\nd(S)$ in \cite{LSZ22}.
Remarkably, the Kuznetsov component is not intrinsic to the surface: different choices of isotropic sequences may produce non-equivalent Kuznetsov components (see \cite[Corollary~2.8]{LSZ22}). 

Further details on these constructions are given in \S\,\ref{ssec:SaturatedIsoSequences}, and explicit examples of non-isomorphic Fano models and non-equivalent Kuznetsov components are given in \S\,\ref{Geom-Fano-models-and-Kuznetsov-components}. 

%---------------------------------------------------------------------------------

\subsection*{Acknowledgements}
We would like to thank Simon Brandhorst, Igor Dolgachev, Dino Festi, Shigeyuki Kond\=o, Gebhard Martin, Margarida Mendes Lopes, Giacomo Mezzedimi, Rita Pardini, Ichiro Shimada, Paolo Stellari, Davide Cesare Veniani, and Xiaolei Zhao for helpful conversations. We also thank the anonymous referee for the valuable comments and suggestions. The first author is a member of GNSAGA of INdAM. During the preparation of the paper, the first author was partially supported by PRIN 2017 Moduli and Lie theory, and by MIUR: Dipartimenti di Eccellenza Program (2018-2022) -
Dept. of Math. Univ. of Pavia. 
The second author is supported by EPSRC grant EP/R034826/1.
While at KTH, the third author was supported by a KTH grant by the Verg foundation.

%---------------------------------------------------------------------------------

\section{Preliminaries}

%---------------------------------------

\subsection{Enriques surfaces and lattices}
\label{Prelim-En-and-latt}

Over an algebraically closed field of characteristic different from $2$, an \emph{Enriques surface} $S$ is a connected smooth projective surface satisfying $2K_S\sim0$ and $h^1(S,\mathcal{O}_S)=h^2(S,\omega_S)=0$. Therefore, $\mathrm{Pic}(S)$ equals the N\'eron--Severi group $\mathrm{NS}(S)$, and after quotienting by the $2$-torsion element $K_S$ we obtain $\Num(S)$, the group of divisors on $S$ modulo numerical equivalence. We have that $\Num(S)$ equipped with the intersection product of curves is a lattice, i.e. a free finitely generated abelian group $L$ equipped with a non-degenerate symmetric bilinear form $b_L\colon L\times L\rightarrow\mathbb{Z}$. As a lattice, $\Num(S)$ is isometric to $U\oplus E_8$, where $U$ denotes the hyperbolic lattice $\left(\mathbb{Z}^2,\left(\begin{smallmatrix}0&1\\1&0\end{smallmatrix}\right)\right)$ and $E_8$ is the negative definite root lattice associated with the corresponding Dynkin diagram.

Given an explicit example of Enriques surface $S$, it will be important for us to find a basis for $\Num(S)$. The idea for this is described in Remark~\ref{remark-how-to-compute-num} below, but before stating it we need some preliminaries. Given a lattice $L$, denote by $L^*$ its dual $\mathrm{Hom}_\mathbb{Z}(L,\mathbb{Z})$. This is naturally identified with
\[
\{v\in L\otimes\mathbb{Q}\mid b_L(v,w)\in\mathbb{Z}~\textrm{for all}~w\in L\}.
\]
As the bilinear form $b_L$ is assumed to be non-degenerate, the assignment $v\mapsto b_L(v,\cdot\;)$ defines an embedding $L\hookrightarrow L^*$, and the quotient $A_L=L^*/L$ is called the \emph{discriminant group} of $L$. If the lattice $L$ is even, which means $b_L(v,v)\in2\mathbb{Z}$ for all $v\in\mathbb{Z}$, then $A_L$ comes equipped with a quadratic form
\begin{align*}
q_L\colon A_L&\rightarrow\mathbb{Q}/2\mathbb{Z},\\
v+L&\mapsto b_L(v,v)~\mathrm{mod}~2\mathbb{Z}
\end{align*}
called the \emph{discriminant quadratic form}. A lattice $M$ containing $L$ as a finite index subgroup is called an \emph{overlattice} of $L$. $M$ gives rise to the isotropic subgroup $M/L$ of $A_L$. More precisely, by \cite[Proposition~1.4.1~(a)]{Nik80} there is a $1$-to-$1$ correspondence between even overlattices of $L$ and subgroups of $A_L$ which are isotropic with respect to $q_L$.

\begin{remark}
\label{remark-how-to-compute-num}
A possible strategy to determine a basis of $\Num(S)$ for an Enriques surface $S$ is the following. Say we have curves $B_1,\ldots,B_{10}$ on $S$ generating a sublattice $L$ of $\Num(S)$ of rank $10$. Then we have that $L=\Num(S)$ if and only if $L$ is unimodular. Otherwise, the elements $x\in\Num(S)\setminus L$ give rise to nonzero classes $x+L\in A_L$ which are isotropic with respect to $q_L$. So one can first list all the isotropic classes $x+L$, and then use the geometry of $S$ to decide which of these satisfy $x\in\Num(S)$.
\end{remark}

%---------------------------------------

\subsection{Elliptic fibrations on Enriques surfaces}
\label{genus-1-pencils-on-En-surf}

We recall the following standard definitions and facts from \cite[Chapter~VIII, \S\,17]{BHPV04} and \cite[\S\,2.2]{CDL22}.

\begin{definition}
Let $f\colon S\rightarrow\mathbb{P}^1$ be an elliptic fibration on an Enriques surface $S$. Then $f$ has exactly two multiple fibers $2F$ and $2F'$. The curves $F$ and $F'$ are called the \emph{half-fibers} of the elliptic fibration $f$.
\end{definition}

We will often use the following standard results concerning half-fibers on Enriques surfaces. By a curve on a surface we mean a connected effective $1$-cycle.

\begin{definition}
Let $S$ be an Enriques surface. An \emph{elliptic configuration on $S$} is a curve $C$ which is primitive in $\Num(S)$ and appears in Kodaira's classification of fibres of elliptic fibrations (see Table \ref{tab_Kodaira}).
\end{definition}

\begin{table}[h!]
\centering\renewcommand\cellalign{cc}
\setcellgapes{3pt}\makegapedcells
\caption{List of fibers of elliptic fibrations, indexed by their intersection graph in the notation of \cite[Chapter~V, Table~3]{BHPV04}. The irreducible components are smooth rational curves, except for the types $\mathrm{I}_0,\mathrm{I}_1,\mathrm{II}$, where the single component is a curve of arithmetic genus 1. Fibers of type $\mathrm{IV}$ only occur for $n=3$.
}
\label{tab_Kodaira}
\begin{tabular}{|c|c|c|c|}
\hline
\makecell{Dynkin \\ notation}
&
\makecell{Kodaira's \\ notation} & \makecell{Irreducible \\ components} &
Dual graph with multiplicities
\\
\hline
$0$ & 
$\mathrm{I}_0,\mathrm{I}_1,\mathrm{II}$ & 1 &
\begin{tikzcd}[every arrow/.append style={dash}, column sep=small, row sep=small, cells={nodes={draw=black, circle,anchor=center}, inner sep=1.5}]
 1
\end{tikzcd}
\\
\hline
$\widetilde{A}_1$ & 
$\mathrm{I}_2,\mathrm{III}$ & 2 &
\begin{tikzcd}[every arrow/.append style={dash}, column sep=small, row sep=small, cells={nodes={draw=black, circle,anchor=center}, inner sep=1.5}]
 1 \arrow[r, shift left=.15em]\arrow[r, shift right=.15em]  & 1 
\end{tikzcd}
\\
\hline
$\widetilde{A}_{n-1}$, ($n\geq 3$) & 
$\mathrm{I}_n, \mathrm{IV}$ & $n$ & 
\begin{tikzcd}[every arrow/.append style={dash}, column sep=small, row sep=small, cells={nodes={draw=black, circle,anchor=center}, inner sep=1.5}]
 &&   1 \ar{dll} \ar{drr}  && \\
  
 1\rar  & 1\rar  & |[draw=none]|\ldots \rar & 1
\rar & 1  
\end{tikzcd}
\\
\hline
$\widetilde{D}_{4+n}$, ($n\geq 0$) & 
$\mathrm{I}^*_n$ & $4+n+1$ &
\begin{tikzcd}[every arrow/.append style={dash}, column sep=small, row sep=small, cells={nodes={draw=black, circle,anchor=center}, inner sep=1.5}]
1 \ar{dr} & & & &  1 \ar{dl}  \\
  & 2 \ar{dl} \rar & |[draw=none]|\ldots \rar & 2 \ar{dr} & \\
1 & &&&1
\end{tikzcd}
\\
\hline
$\widetilde{E}_{6}$ & 
$\mathrm{IV}^*$ & 7 &
\begin{tikzcd}[every arrow/.append style={dash}, column sep=small, row sep=small, cells={nodes={draw=black, circle,anchor=center}, inner sep=1.5}]
 & & 1 \dar & &   \\
 & & 2 \dar & & \\
1 \rar &2\rar &3\rar &2\rar & 1
\end{tikzcd}
\\
\hline
$\widetilde{E}_{7}$ & 
$\mathrm{III}^*$ & 8 &
\begin{tikzcd}[every arrow/.append style={dash}, column sep=small, row sep=small, cells={nodes={draw=black, circle,anchor=center}, inner sep=1.5}]
 & & & 2 \dar & & & \\
1 \rar &2\rar &3\rar &4 \rar &3\rar &2\rar & 1
\end{tikzcd}
\\
\hline
$\widetilde{E}_{8}$ & 
$\mathrm{II}^*$ & 9 &
\begin{tikzcd}[every arrow/.append style={dash}, column sep=small, row sep=small, cells={nodes={draw=black, circle,anchor=center}, inner sep=1.5}]
 & & 3 \dar & & &&& \\
2 \rar &4\rar &6\rar &5 \rar &4\rar &3\rar &2\rar & 1
\end{tikzcd}
\\
\hline
\end{tabular}
\end{table}

\begin{remark}
\label{rmk:rankOfAnDn}
If the dual graph of $C_{\text{red}}$ is  $\widetilde{A}_n$ or $\widetilde{D}_n$, then we must have that $n\leq8$ as $\Num(S)$ has signature $(1,9)$.
\end{remark}

\begin{lemma}
\label{ellipticconfigurationgivesgenusonepencil}
Let $C$ be an elliptic configuration on an Enriques surface. Then either $|C|$ is an elliptic pencil or $|2C|$ is an elliptic pencil of which $C$ is one of the two half-fibers.
\end{lemma}

\begin{lemma}
\label{tooltodistinguishfiberfromhalf-fiber}
Let $S$ be an Enriques surface and let $f\colon S\rightarrow\mathbb{P}^1$ be an elliptic fibration. Let $F_1,F_2$ be the half-fibers and $F$ a reduced fiber of $f$. Let $\pi\colon X\rightarrow S$ be the universal K3 cover of $S$. Then $\pi^{-1}(F_1),\pi^{-1}(F_2)$ are connected and $\pi^{-1}(F)$ is disconnected.
\end{lemma}

\begin{lemma}[{\cite[Chapter~V, Theorem~5.7.5~(i)]{CD89}}]
\label{possiblehalf-fibersincharacteristiczero}
Let $F$ be a half-fiber on an Enriques surface. Then $F$ is of type $\widetilde{A}_n$ for $1\leq n\leq8$ or a smooth genus one curve. In particular, if an elliptic configuration $C$ has dual graph $\widetilde{D}_n$ or $\widetilde{E}_n$, then $C$ is a fiber.
\end{lemma}

\subsection{Isotropic sequences and the non-degeneracy invariant}

Here we recall some preliminary notions and the definition of the non-degeneracy invariant, as it was given in the introduction. We follow \cite[Chapter~III]{CD89}.

\begin{definition}
An \textit{isotropic sequence} is a sequence of primitive isotropic vectors $(e_1,\ldots,e_n)$ in $\Num(S)$ satisfying $e_i\cdot e_j=1-\delta_{ij}$. Additionally, $(e_1,\ldots,e_n)$ is called \textit{non-degenerate} if every $e_i$ is the class of a nef divisor, and \textit{maximal} if $n=10$.
\end{definition}

\begin{remark}
Note that if $e\in\Num(S)\setminus\{0\}$ is the class of a nef divisor $E$ and $e^2=0$, then $E$ must be effective. To prove this, first observe by Riemann--Roch that $E$ or $K_S-E$ is effective, but not both. If by contradiction $K_S-E$ is effective, then one can show that $K_S-E$ is numerically trivial, which implies $e=0$.
\end{remark}

\begin{remark}
If $E_1,\ldots,E_n$ are half-fibers whose classes $e_i$ satisfy \eqref{eq:FiFj=1-dij}, then $(e_1,\ldots,e_n)$ is a non-degenerate isotropic sequence in $\Num(S)$. It is a standard fact that the converse also holds, however we briefly review its proof for the interested reader. 

Suppose $(e_1,\ldots,e_n)$ is a non-degenerate isotropic sequence, so that each $e_i$ is the class of a nef divisor $E_i$. First note that $E_i$ intersects all of its components $C$ trivially: as $E_i$ is nef, $E_i\cdot C\geq0$ and $E_i\cdot(E_i-C)\geq0$, so $0\leq E_i\cdot C\leq0$. Let $C_{i1},\ldots,C_{i\ell}$ be the connected components of $E_i$, and write $C_{ij}=m_{ij}C_{ij}'$ for some positive integer $m_{ij}$ and a curve $C_{ij}'$ with primitive class. Then the $C_{ij}'$ are \emph{indecomposable} \cite[Chapter~III, \S\,1]{CD89}, and using \cite[Proposition~3.1.1]{CD89} we can see that $C_{ij}'$ is an elliptic configuration. So, Lemma~\ref{ellipticconfigurationgivesgenusonepencil} combined with the fact that $[C_{ij}']$ is primitive imply that $|2C_{ij}'|$ is an elliptic pencil of which $C_{ij}'$ is a half-fiber. As the $C_{ij}'$ are disjoint, they are numerically equivalent, implying that $e_i=(\sum_{j=1}^\ell m_{ij})[C_{i1}]$. As $e_i$ is primitive, the only possibility is that $\ell=1$ and $m_{11}=1$. So $E_i$ is connected, and it is the half-fiber of an elliptic pencil.
\end{remark}

\begin{definition}
\label{defn:ndinvariant}
Let $S$ be an Enriques surface. Define the \emph{non-degeneracy invariant of $S$}, denoted by $\nd(S)$, as the maximum integer $n$ such that there exists a non-degenerate isotropic sequence of length $n$. Equivalently, $\nd(S)$ is the maximum $n$ for which there exist $F_1,\ldots,F_n$ half-fibers on $S$ such that $F_i\cdot F_j=1$ for all $i\neq j$.
\end{definition}

It is possible to give a geometric interpretation to degenerate isotropic sequences as well. Since two distinct smooth rational curves on $S$ cannot be numerically equivalent, we can identify the set $\sR(S)$ of smooth rational curves on $S$ with the subset of $\Num(S)$ given by their classes. Moreover, every $R\in \sR(S)$ satisfies $R^2=-2$ and intersects all the other $R'\in \sR(S)$ non-negatively. Therefore, $\sR(S)$ is a set of roots of $\Num(S)$. The associated Weyl group $W$ acts on $\Num(S)$ by reflections across elements of $\sR(S)$. Every $W$-orbit of an isotropic sequence in $\Num(S)$ admits a (unique) representative, called \textit{canonical}, which is geometrically meaningful:

\begin{lemma}[{\cite[Lemma~3.3.1]{CD89}}, {\cite[Proposition~6.1.5]{DK22}}]
\label{lem:CanonicalSequences}
Suppose that $(f_1,\ldots,f_k)$ is an isotropic sequence in $\Num(S)$. Then there is a unique $w\in W$ such that, up to reordering:
\begin{itemize}
    \item the sequence $(f_1',\ldots,f_k')\coloneqq (w(f_1),\ldots,w(f_k))$ contains a non-degenerate subsequence $(f_{i_1}',\ldots,f_{i_c}')$ with $1=i_1< \ldots <i_c$;
    \item for any $i_s<i<i_{s+1}$ there are rational curves $R^{i_s}_1,\ldots,R^{i_s}_{i-i_s}$ such that 
    \[ f_i'=f_{i_s}' + R^{i_s}_1 + \ldots + R^{i_s}_{i-i_s} \in W\cdot f_{i_s}.\]
    Here, $R^{i_s}_1 + \ldots + R^{i_s}_{i-i_s}$ is a chain of type $A_{i-i_s}$.
\end{itemize}
\end{lemma}

Any sequence which up to reordering has the form  $(f_1',\ldots,f_k')$ is called a \textit{canonical} isotropic sequence. Its \textit{non-degeneracy} is the number $c$ of nef classes it contains. (Observe that by our definition all non-degenerate sequences are canonical. This is a slight discrepancy with \cite[Chapter~III, \S\,3]{CD89}, but it should not cause confusion.) We conclude this section with the following result about extensions of non-degenerate sequences.

\begin{lemma}[{\cite[Corollary~3.3.1]{CD89}}]
\label{lem:ExtendingSequence}
Let $k\neq 9$. Then every non-degenerate isotropic sequence $(f_1,\ldots,f_k)$ can be extended to a canonical maximal isotropic sequence $(f_1,\ldots,f_k,f_{k+1},\ldots,f_{10})$ of non-degeneracy $c\geq k$.
\end{lemma}

\begin{remark}
The extension $(f_1,\ldots,f_k,f_{k+1},\ldots,f_{10})$ in Lemma~\ref{lem:ExtendingSequence} is in general not unique, as illustrated in Example~\ref{ex_Kondo1Extensions}.
\end{remark}

%---------------------------------------

\section{A combinatorial version of the non-degeneracy invariant of Enriques surfaces}

\subsection{The combinatorial non-degeneracy invariant}

We now introduce a purely combinatorial version of the non-degeneracy invariant, which applied to Enriques surfaces yields a lower bound for $\nd(S)$.

\begin{definition}
Let $G=(V,E,w)$ be a finite, undirected, simple graph with vertices $V=\{v_1,\ldots,v_k\}$, edges $E$, and a weight function $w\colon E\rightarrow \mathbb{Z}_{>0}$. Let $L_G=\oplus_{i=1}^n\mathbb{Z}v_i$. An element $x=\sum_i a_iv_i\in L_G$ will be called an \emph{elliptic vector} if it satisfies the following conditions:
\begin{enumerate}

\item the vertices $v_i$ with $a_i\neq0$ induce a subgraph of $G$ which is an extended Dynkin diagram of type $\widetilde{A}_n$, $\widetilde{D}_n$, or $\widetilde{E}_6,\widetilde{E}_7,\widetilde{E}_8$;

\item the nonzero coefficients $a_i$ are as in Kodaira's classification of singular fibers of elliptic fibrations.

\end{enumerate}
We can endow $L_G$ with a symmetric bilinear form $b_G$ obtained by extending the following:
\begin{displaymath}
b_G(v_i,v_j)=\left\{ \begin{array}{ll}
-2~&\textrm{if}~i=j\\
0~&\textrm{if}~i\neq i~\textrm{and}~(v_i,v_j)\notin E\\
w(v_i,v_j)~&\textrm{if}~i\neq j~\textrm{and}~(v_i,v_j)\in E.
\end{array} \right.
\end{displaymath}
If we let $\mathrm{Null}(b_G)=\{x\in L_G\mid b_G(x,y)=0~\textrm{for all}~y\in L_G\}$, then $\overline{L}_G=L_G/\mathrm{Null}(b_G)$ is a free $\mathbb{Z}$-module and $b_G$ induces on it a well defined non-degenerate symmetric bilinear form, making $\overline{L}_G$ into a lattice. Let $N$ be a fixed overlattice of $\overline{L}_G$. For an elliptic vector $v\in L_G$, define $c_N([v])=\frac{1}{2}[v]$ if $\frac{1}{2}[v]\in N$ and $c_N([v])=[v]$ otherwise. Let
\[
\mathsf{HF}(G,N)=\{c_N([v])\mid v\in L_G~\textrm{is an elliptic vector}\}\subseteq N.
\]
Then we define the \emph{combinatorial non-degeneracy invariant} $\cnd(G,N)$ to be the maximum $m$ such that there exist $f_1,\ldots,f_m\in\mathsf{HF}(G,N)$ satisfying $f_i\cdot f_j=1-\delta_{ij}$.
\end{definition}

\begin{proposition}
Let $S$ be an Enriques surface and let $\mathcal{R}=\{R_1,\ldots,R_k\}$ be a finite collection of smooth rational curves on $S$. Let $G$ be the graph dual to the configuration $\mathcal{R}$ with weights given by the intersection numbers $R_i\cdot R_j$ for $i\neq j$. Then $\cnd(G,\Num(S))\leq\nd(S)$.
\end{proposition}

\begin{proof}
By construction, we have that the elliptic vectors in $L_G$ are classes of elliptic configurations on $S$ and $\mathsf{HF}(G,\Num(S))$ is a collection of classes of half-fibers on $S$. From this we obtain the claimed inequality, because $\nd(S)$ considers all the half-fibers on $S$, while $\cnd(G,\Num(S))$ only the ones in $\mathsf{HF}(G,\Num(S))$.
\end{proof}

\begin{definition}
Let $S$ be an Enriques surface and $\mathcal{R}$ a finite collection of smooth rational curves $\mathcal{R}=\{R_1,\ldots,R_k\}$ on $S$ with dual graph $G$. We define $\mathcal{E}(S,\sR)$ as the set of elliptic fibrations $|2F|$ on $S$ for $F\in\mathsf{HF}(G,\Num(S))$. Moreover, in this case we denote $\mathsf{HF}(G,\Num(S))$ and $\cnd(G,\Num(S))$ simply by $\mathsf{HF}(S,\mathcal{R})$ and $\cnd(S,\mathcal{R})$.
\end{definition}

\begin{remark}
Notice that if $\mathcal{E}(S,\mathcal{R})$ contains all the elliptic fibrations on $S$, then the combinatorial non-degeneracy invariant $\cnd(S,\mathcal{R})$ equals $\nd(S)$.
\end{remark}

\begin{remark}
\label{rmk:catFiberHalfFiber}
Suppose we have an Enriques surface $S$ and a finite collection $\mathcal{R}$ of smooth rational curves on it. To determine $\cnd(S,\mathcal{R})$ we first determine the set $\mathsf{HF}(S,\mathcal{R})$. So, for an elliptic configuration $C$ with irreducible components in $\mathcal{R}$, it will be important to distinguish whether $C$ is either a fiber or a half-fiber of an elliptic fibration (these are the only possibilities by Lemma~\ref{ellipticconfigurationgivesgenusonepencil}). We have two strategies:
\begin{enumerate}

\item Apply Lemma~\ref{tooltodistinguishfiberfromhalf-fiber} to the universal K3 cover of $S$.

\item Say we have a basis $\{B_1,\ldots,B_{10}\}$ of $\Num(S)$. As the lattice $\Num(S)$ is unimodular, if $(B_i\cdot C)/2$ is an integer for all $i$, then $C$ is a fiber. Otherwise, $C$ is a half-fiber.

\end{enumerate}
\end{remark}

Therefore, given $S$, $\sR$, and either the universal cover of $S$ or a basis for $\Num(S)$, the problem of evaluating $\cnd(S,\mathcal{R})$ can be automatized with a computer. We implement this in \S\,\ref{section-describing-code}.

%----------------------------------------------------------------------------------------------

\subsection{Saturated isotropic sequences}
\label{ssec:SaturatedIsoSequences}

\begin{definition}
\label{defn:saturated}
A non-degenerate isotropic sequence $(f_1,\ldots,f_k)$ is \textit{not saturated} if it can be extended to a non-degenerate isotropic sequence of length $c>k$. It is called \textit{saturated} otherwise.
\end{definition}

We also introduce a relative notion of saturatedness, for which we fix a collection $\mathcal{R}$ of smooth rational curves on $S$. 

\begin{definition}
\label{defn:Rsaturated}
Let $(f_1,\ldots,f_k)$ be a non-degenerate isotropic sequence of classes in $\mathsf{HF}(S,\sR)$. Then, we say that $(f_1,\ldots,f_k)$ is \textit{not} $\sR$-\textit{saturated} if it can be extended to a non-degenerate isotropic sequence of length $c>k$ by adding classes in $\mathsf{HF}(S,\sR)$. It is called $\sR$\textit{-saturated} otherwise.
\end{definition}

These definitions are motivated by the fact that saturated sequences in combination with Lemma~\ref{lem:ExtendingSequence} can be used to produce examples of non-isomorphic Fano models and non-equivalent Kuznetsov components of $S$. Let us first recall these concepts. Suppose that $(f_1,\ldots,f_c)$ is a non-degenerate isotropic sequence which is saturated. If $c\neq9$, then by Lemma~\ref{lem:ExtendingSequence} we can extend it to a maximal canonical isotropic sequence $(f_1,\ldots,f_c,f_{c+1},\ldots,f_{10})$ of non-degeneracy still equal to $c$. This means that, after appropriately reordering $f_1,\ldots,f_{10}$, there exist indices $i_1,\ldots,i_c$ such that $f_{i_1},\ldots,f_{i_c}$ are classes of half-fibers, and $f_i$ for $i_s<i<i_{s+1}$ has the form
\[
f_i =f_{i_s} + R^{i_s}_1 + \ldots + R^{i_s}_{i-i_s},
\]
where the $R^{i_s}_1+\ldots+R^{i_s}_{i_{s+1}-i_s-1}$ form a chain of type $A_{i_{s+1}-i_s-1}$ (see Lemma \ref{lem:CanonicalSequences}). As mentioned in the introduction, the vector $\frac{1}{3}(f_1+\ldots+f_{10})\in \Num(S)$ is the class of a nef divisor $\Delta$ called a Fano polarization. The linear series $|\Delta|$ maps $S$ to a normal surface of degree 10 in $\mathbb{P}^5$, called a Fano model of $S$. This morphism contracts exactly the rational curves of class $R^{i_s}_k$, $k=1,\ldots,i_{s+1}-i_s-1$, giving rise to singularities of type $A_{i_{s+1}-i_s-1}$. $\Delta$ is very ample if and only if all the $f_i$ are classes of half-fibers. In other words, $S$ admits a very ample Fano polarization if and only if $\nd(S)=10$ (we refer the interested reader to the discussion in \cite[\S\,2.3]{DM19}). 

From the point of view of derived categories, one can use $(f_1,\ldots,f_{10})$ as above to construct a subcategory of the bounded derived category $D^b(\Coh(S))$ as follows. Let $F_{i_s}$, $1\leq s \leq c$, denote one of the half-fibers of the fibrations corresponding to $f_{i_s}$. For $i_s<i<i_{s+1}$, define $ F_i =F_{i_s} + R^{i_s}_1 + \ldots + R^{i_s}_{i-i_s} $. We have that $\mathcal{L}=(\sO(F_1),\ldots,\sO(F_{10}))$ is an exceptional collection \cite[Proposition~3.5]{LNSZ21} whose orthogonal complement $\Ku(S,\mathcal{L})$ is called a Kuznetsov component of $D^b(\Coh(S))$.

Now, suppose that $Q_1,Q_2$ are two saturated sequences of length $c_1\neq c_2$, with $c_1\neq 9 \neq c_2$. By Lemma~\ref{lem:ExtendingSequence}, $Q_1$ and $Q_2$ can be extended to canonical maximal isotropic sequences $P_1,P_2$ of non-degeneracy $c_1,c_2$ respectively. For $\ell=1,2$, $P_\ell$ defines a Fano polarization $\Delta_\ell$ and a Fano model $S_\ell$. The singularities of $S_\ell$ are determined by the curves contracted by $\Delta_\ell$, which are precisely the rational curves appearing among the vectors of $P_\ell$, and there are $10-c_\ell$ of such smooth rational curves. Since $c_1\neq c_2$, we have that $S_1$ and $S_2$ have different singularities, so they cannot be isomorphic.

Similarly, $P_\ell$ defines an exceptional collection $\sL_\ell$ and a Kuznetsov component $\Ku(S,\sL_\ell)$. As shown in \cite[Theorem~2.7]{LSZ22}, up to shifts and isomorphism there are exactly $c_\ell$ objects in $\Ku(S,\sL_\ell)$ that are $3$-spherical or $3$-pseudoprojective. Again, since $c_1\neq c_2$, we conclude that $\Ku(S,\sL_1)\not\simeq \Ku(S,\sL_2)$. The same strategy is used in \cite[Corollary~2.8]{LSZ22} to show that general nodal Enriques surfaces always admit non-equivalent Kuznetsov components.

Explicit examples of the scenarios above are discussed in \S\,\ref{Geom-Fano-models-and-Kuznetsov-components}.

%---------------------------------------------------------------------------------

\section{A SageMath code for computing the non-degeneracy invariant}
\label{section-describing-code}

In this section we present the SageMath code \texttt{CndFinder}, available at \cite{MRS22}, which computes the set $\mathsf{HF}(S,\mathcal{R})$ and consequently determines the combinatorial non-degeneracy invariant $\cnd(S,\mathcal{R})$ for an Enriques surface $S$ and a collection of smooth rational curves $\mathcal{R}$ on $S$.

\subsection{Notation}
In what follows and in the code, the objects involved in the computation of the combinatorial non-degeneracy invariant are categorized according to their \textit{type}. Here we make this notion precise and fix some notation. In particular, we define the type of an elliptic configuration, of an elliptic fibration, and of an isotropic sequence.

In the code, we denote extended Dynkin diagrams with just their letter and rank. For instance \texttt{D8} refers to $\widetilde{D}_8$. The type of an elliptic configuration is the associated Dynkin diagram, together with the information of being a fiber or a half-fiber. For example, \texttt{A7HF} refers to an elliptic configuration whose underlying diagram is $\widetilde{A}_7$ and which is a half-fiber. Throughout the paper, we use the more compact notation $\widetilde{A}_7^{\hf}$. 

Within the code, the type of an elliptic fibration is the formal sum of the types of its singular fibers supported in $\mathcal{R}$. For instance, in the code, \texttt{(2~A1HF~+~1~D6F)} refers to the fibrations whose singular fibers are three elliptic configurations, two of type $\widetilde{A}_1^{\hf}$ and one of type $\widetilde{D}_6^{\f}$. Throughout the paper, we use the more compact notation $(2\widetilde{A}_1^{\hf}+\widetilde{D}_6^{\f})$. 

Finally, the type of a non-degenerate isotropic sequence is the list of the types of the elliptic fibrations appearing in it. So sequences of type
\begin{center}
\texttt{4~x~(1~A1F~+~1~A7F)},~\texttt{1~x~(2~A1F~+~2~A3HF)},~\texttt{1~x~(2~A1HF~+~1~D6F)}
\end{center}
contain one half-fiber of each of four fibrations of type $(\widetilde{A}_1^{\f} + \widetilde{A}_7^{\f})$, one half-fiber of a fibration of type $(2\widetilde{A}_1^{\f} + 2\widetilde{A}_3^{\hf})$, and one half-fiber of a fibration of type $(2\widetilde{A}_1^{\hf}+\widetilde{D}_6^{\f})$. Throughout the paper, we use the more compact notation
\[
4\times(\widetilde{A}_1^{\f} + \widetilde{A}_7^{\f}),(2\widetilde{A}_1^{\f} + 2\widetilde{A}_3^{\hf}),(2\widetilde{A}_1^{\hf}+\widetilde{D}_6^{\f}).
\]

\subsection*{Input} The input required is a collection $\sR =\{R_1, \ldots, R_k\}$ of smooth rational curves which span $\Num(S)$ over $\mathbb{Q}$, together with a basis of $\Num(S)$ consisting of $\mathbb{Q}$-linear combinations of curves in $\sR$. The following command starts the calculation, saving all the data in the variable named \texttt{FinalResult}.

\begin{lstlisting}
from nd_sequences_finder import *
IntersectionMatrix=matrix([[...]])
BasisNum=[[...]]
FinalResult=CndFinder(IntersectionMatrix,BasisNum)
\end{lstlisting}
Here, \texttt{IntersectionMatrix} is the $k \times k$ intersection matrix of $\sR$. \texttt{BasisNum} is an array which specifies a basis of $\Num(S)$, written in terms of the generating set $\sR$. 

\subsection*{The main algorithm} 

The code proceeds as follows:

\subsubsection*{(Step 1)}

The code identifies the elliptic configurations supported on $\sR$, grouped according to their type. As we start with a collection of smooth rational curves on $S$, the possible types that can arise are $\widetilde{A}_1, \ldots, \widetilde{A}_8$, $\widetilde{D}_4, \ldots, \widetilde{D}_8$, $\widetilde{E}_6, \widetilde{E}_7, \widetilde{E}_8$ by Remark~\ref{rmk:rankOfAnDn}.  If $N$ denotes the intersection matrix of an extended Dynkin diagram as above, then the code lists all the subsets $X\subseteq \sR$ whose intersection matrix equals $N$. 

Note that $\widetilde{A}_1$ and $\widetilde{A}_2$ admit two distinct geometric realizations each, but their intersection matrices coincide. The code cannot distinguish between them, but this does not affect the end result for $\cnd(S,\mathcal{R})$.

\textbf{The output of step 1.} For each extended Dynkin diagram $N$ as above, this step lists all the subsets $X_i \subseteq \sR$ with intersection matrix $N$. We say that the $X_i$ have type given by the Dynkin diagram associated to $N$.
The code then groups the $X_i$ together according to their type.

\subsubsection*{(Step 2)}
By construction, there is a unique elliptic configuration $C_i$ supported on $X_i$. 
By Lemma~\ref{ellipticconfigurationgivesgenusonepencil}, either $C_i$ or $\frac{1}{2}C_i$ is primitive in $\Num(S)$. 
To decide this, the code applies strategy (2) of Remark~\ref{rmk:catFiberHalfFiber}. First, it assumes $C_i$ is not primitive in $\Num(S)$, and stores in memory the array of coefficients of $1/2*C_i$. Then, the code decides whether $\frac{1}{2}[C_i] \in \Num(S)$ by intersecting $1/2*C_i$ with every element of \texttt{BasisNum}. If all the intersections are integers, then $\frac{1}{2}[C_i]\in \Num(S)$. 
Otherwise, $[C_i]$ is primitive in $\Num(S)$, and the code replaces $1/2*C_i$ with $C_i$. 
This is repeated for each subset $X_i\subseteq \sR$ obtained in the previous step.

\textbf{The output of step 2} is the list $\{C_1,\ldots,C_n\}$, where $[C_i]$ is the unique class of a half-fiber associated with $X_i$. The $C_i$ are grouped together according to their type.

\subsubsection*{(Step 3)}

The curves $\{C_1,\ldots,C_n\}$ from step 2 may satisfy $|2C_i|=|2C_j|$ for $i\neq j$. This happens if and only if $C_i\cdot C_j=0$. Step 3 eliminates the redundancy and lists distinct elliptic fibrations.

\textbf{The output of step 3} is the list of elements of $\sE(S,\sR)$ and $\mathsf{HF}(S,\mathcal{R})=\{[F_1],\ldots,[F_m]\}$, together with the choice of the representative $F_i$ for each class $[F_i]$. This information is saved in the key \texttt{EllipticFibrations} in the output dictionary. Strictly speaking, this step is not necessary to compute $\cnd(S,\sR)$, but it arranges the data in a more geometrically meaningful way and it speeds up the computation significantly. The elliptic fibrations are grouped together depending on their type.

\subsubsection*{(Step 4)}
For each type $T_i$ of elliptic fibration which was found in the previous step, this step computes an integer $m_i=1, \dots, 10$. The number $m_i$ equals the maximum number of elliptic fibrations of type $T_i$ that can appear in the same isotropic sequence. Like step 3, step 4 is not strictly necessary to compute $\cnd(S,\sR)$, but it improves the computing time.

\textbf{The output of step 4} is the same as the output of step 3, with the additional information of the numbers $m_i$ associated with each type of elliptic fibration.

\subsubsection*{(Step 5)}
This is a recursive step. Roughly, the code starts with a (initially empty) list $\mathbb L$ of isotropic sequences, and tries to add to each sequence an element of $\mathsf{HF}(S,\mathcal{R})$. Afterward, the code calls the function again, and it stops when extending sequences in $\mathbb{L}$ is no longer possible. This is described in more detail below.

More precisely, a class $[F_i]\in \mathsf{HF}(S,\mathcal{R})$ can be added to an isotropic sequence $([F_{i_1}], \ldots, [F_{i_t}])$ if and only if 
$([F_{i_1}], \ldots, [F_{i_t}], [F_i])$ satisfies \eqref{eq:FiFj=1-dij}. 
To check this condition efficiently, we introduce an ordering on the set $\mathsf{HF}(S,\sR)$ based on the type of half-fiber classes.

The possible types $\{T_1, \ldots, T_r\}$ define a partition of $\mathsf{HF}(S,\sR)$: for $i=1,\ldots,r$ let $\{[F_{j}^{(i)}]\}_{j=1}^{n_i}$ be the set of elements of $\mathsf{HF}(S,\sR)$ of type $T_i$. Given $F^{(i)}_j,F^{(i')}_{j'}\in \mathsf{HF}(S,\sR)$, we declare that $F^{(i)}_j>F^{(i')}_{j'}$ provided $i>i'$, or $i=i'$ and $j>j'$.

The isotropic sequences in $\mathbb L$ are in increasing order. Suppose that an (ordered) isotropic sequence ends with the class $[F_{\bar j}^{(\bar i)}]$. Then the code tries to add to it all the elements $F_{j}^{(\bar i)}$ in $T_{\bar i}$, with $j>\bar j$, and all the elements $F_j^{(i)}$ in $T_i$ with $i> \bar i$. If a class is successfully added to the sequence, the extended sequence is added to $\mathbb{L}$, and the function is called again. Otherwise, the recursion stops.

\textbf{The output of step 5} is the list $\mathbb{L}$ of all isotropic sequences of elements in $\mathsf{HF}(S,\mathcal{R})$. In particular, the longest sequences in $\mathbb{L}$ have length equal to $\cnd(S,\mathcal{R})$.

\subsubsection*{(Step 6)}  If an isotropic sequence $Q\in \mathbb{L}$ is not $\sR$-saturated, there is another $Q'\in \mathbb{L}$ containing all the elements of $Q$. In this case, $Q$ is discarded.

\textbf{The output of step 6} is the list of $\sR$-saturated sequences. In the output dictionary, it is saved in the key \texttt{SaturatedSequences}.

\begin{remark}[Characteristic $2$]
\label{rmk:programCharacteristic2}
The code produces the correct $\cnd(S,\mathcal{R})$ also for Enriques surfaces in characteristic $2$. First of all, if $C$ is an elliptic configuration, $|C|$ or $|2C|$ is an elliptic pencil by \cite[Theorem~2.2.8]{CDL22}. Moreover, if an elliptic or quasi-elliptic fibration on an Enriques surface has a multiple fiber, that multiplicity is $2$. The reason why we kept Enriques surfaces in characteristic $2$ separate from our discussion is because for these the non-degeneracy invariant $\nd(S)$ behaves very differently. For instance, there exist Enriques surfaces in characteristic $2$ which satisfy $1\leq\nd(S)\leq3$, and all three possibilities occur (see \cite[Chapter~6]{DK22} and \cite{MMV22b}).
\end{remark}

\begin{remark}
Let $D$ be a big divisor on an Enriques surface $S$. The function
\[
\Phi(D)=\min\{D\cdot F\mid F~\textrm{is a half-fiber on}~S\}
\]
(see \cite[Equation~(2.4.7)]{CDL22}) encodes information about the linear system $|D|$. For instance, if $D$ is also nef, then $\Phi(D)=1$ if and only if $|D|$ has at least one base point \cite[Theorem~2.4.14]{CDL22}. We refer the reader to \cite[\S\,2.4--\S\,2.6]{CDL22} for a general discussion. One can define a version of this invariant which is relative to a configuration $\mathcal{R}$ of finitely many smooth rational curves on $S$. More precisely, we call the \emph{combinatorial $\Phi$-invariant of $D$ with respect to $\mathcal{R}$} the minimum of $D\cdot F$ as $[F]\in\mathsf{HF}(S,\mathcal{R})$. The calculation of the combinatorial $\Phi$-invariant is then a variation of \texttt{CndFinder} (we thank the referee for suggesting this), which is also available at \cite{MRS22}. The function is called by the following command:
\begin{lstlisting}
CPhiFinder(IntersectionMatrix,BasisNum,DivisorsList)
\end{lstlisting}
The input is the same as \texttt{CndFinder}, with the addition of a list of (big) divisors $D_1,\ldots,D_\ell$ of which we want to compute the combinatorial $\Phi$-invariant. Each $D_i$ is specified as a linear combination with rational coefficients of the smooth rational curves in $\mathcal{R}$.
\end{remark}

%--------------------------------------------------------------------------------

\section{Enriques surfaces which are \texorpdfstring{$\mathbb{Z}_2^2$}{Lg}-cover of \texorpdfstring{$\mathbb{P}^2$}{Lg}}
\label{sec:D16}

We now begin our series of examples of Enriques surfaces where we apply the code described in \S\,\ref{section-describing-code}. For simplicity, we work over $\mathbb{C}$.

\begin{definition}
Consider the blow up of $\mathbb{P}^2$ at three not-aligned points $\Bl_3\mathbb{P}^2$, which comes with three distinct rulings $\pi_i\colon\Bl_3\mathbb{P}^2\rightarrow\mathbb{P}^1$, $i=1,2,3$. For each ruling $\pi_i$, choose two distinct fibers $\ell_i,\ell_i'$ which are smooth lines, so that the overall arrangement $\{\ell_1,\ell_1',\ell_2,\ell_2',\ell_3,\ell_3'\}$ of six lines on $\Bl_3\mathbb{P}^2$ does not have triple intersection points. Write $\mathbb{Z}_2^2=\{e,a,b,c\}$, where $e$ is the identity element. Let $S\rightarrow\Bl_3\mathbb{P}^2$ be the $\mathbb{Z}_2^2$-cover with the following building data \cite[Definition~2.1]{Par91}:
\[
D_a=\ell_1+\ell_1',~D_b=\ell_2+\ell_2',~D_c=\ell_3+\ell_3'.
\]
One can verify using tools in \cite{Par91} that $S$ is an Enriques surface (see \cite[Definition~2.1]{Sch22} for details). Adopting the same name introduced in \cite{Oud11}, we call $S$ a \emph{$D_{1,6}$-polarized Enriques surface}. $D_{1,6}$ denotes the sublattice of $\langle-1\rangle\oplus\langle1\rangle^{\oplus6}$ of vectors with even square, and the above Enriques surface $S$ admits a primitive embedding of $D_{1,6}$ into $\mathrm{Pic}(S)$ satisfying specific geometric properties (see \cite[\S\,3.1]{Oud11} for details). We will not need such a lattice-theoretic characterization, and the covering construction given will suffice for our purposes.
\end{definition}

\begin{remark}
Compactifications of the moduli space of $D_{1,6}$-polarized Enriques surfaces were studied in \cite{Oud11,Sch22}. The universal K3 covers of the $D_{1,6}$-polarized Enriques surfaces were studied in \cite{Sch18} from the point of view of their automorphisms.
\end{remark}

\begin{lemma}
\label{computation12nodesond16enriquessurface}
Let $S$ be a $D_{1,6}$-polarized Enriques surface and let $S\rightarrow\Bl_3\mathbb{P}^2$ be the corresponding $\mathbb{Z}_2^2$-cover. Then the preimage of the six $(-1)$-curves in $\Bl_3\mathbb{P}$ gives a configuration of $(-2)$-curves whose dual graph is in Figure~\ref{D16-En-dualgraph12(-2)-curves}.
\end{lemma}

\begin{proof}
The $\mathbb{Z}_2^2$-cover $S\rightarrow\Bl_3\mathbb{P}^2$ can be realized as the composition of two double covers $S\rightarrow S'\rightarrow\Bl_3\mathbb{P}^2$: the first double cover is branched along $\ell_1+\ell_1'+\ell_2+\ell_2'$, and the second one is branched along the preimage of $\ell_3+\ell_3'$ and the four $A_1$ singularities of $S'$. The preimage of the six $(-1)$-curves in $\Bl_3\mathbb{P}^2$ is computed step by step in Figure~\ref{Z22coverand12(-2)-curves}, and on the right we can see the resulting configuration on the Enriques surface $S$.
\end{proof}

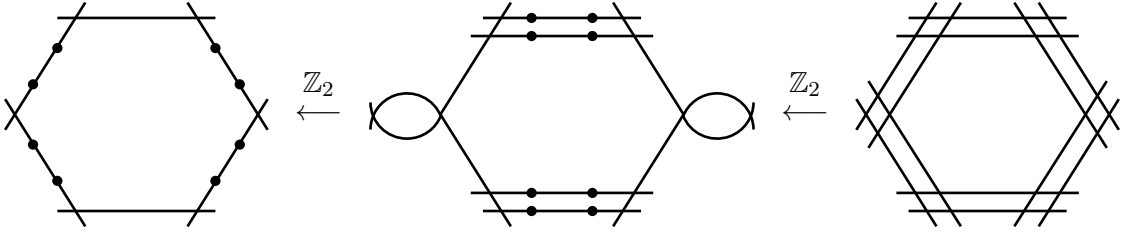
\begin{figure}
\begin{tikzpicture}[scale=0.8]

	\draw[line width=1.0pt] (-1.30000000000000,1.60000000000000) -- (1.30000000000000,1.60000000000000);
	\draw[line width=1.0pt] (0.841000317999046,1.85439949120153) -- (2.15899968200095,-0.254399491201526);
	\draw[line width=1.0pt] (2.15899968200095,0.254399491201526) -- (0.841000317999046,-1.85439949120153);
	\draw[line width=1.0pt] (1.30000000000000,-1.60000000000000) -- (-1.30000000000000,-1.60000000000000);
	\draw[line width=1.0pt] (-0.841000317999046,-1.85439949120153) -- (-2.15899968200095,0.254399491201526);
	\draw[line width=1.0pt] (-2.15899968200095,-0.254399491201526) -- (-0.841000317999046,1.85439949120153);

	\fill (-1.7,0.5) circle (2.4pt);
	\fill (-1.3,1.1) circle (2.4pt);

	\fill (1.7,0.5) circle (2.4pt);
	\fill (1.3,1.1) circle (2.4pt);

	\fill (-1.7,-0.5) circle (2.4pt);
	\fill (-1.3,-1.1) circle (2.4pt);

	\fill (1.7,-0.5) circle (2.4pt);
	\fill (1.3,-1.1) circle (2.4pt);

	\node at (3,0.5) {$\mathbb{Z}_2$};
	\node at (3,0) {$\longleftarrow$};

%-----------------------------------------------

	\draw[line width=1.0pt] (5.70000000000000,1.60000000000000) -- (8.30000000000000,1.60000000000000);
	\draw[line width=1.0pt] (5.50000000000000,1.30000000000000) -- (8.50000000000000,1.30000000000000);

	\draw[line width=1.0pt] (7.84100031799905,1.85439949120153) -- (9,0);
	\draw[line width=1.0pt] (9,0) -- (7.84100031799905,-1.85439949120153);
	\draw[line width=1.0pt] (8.30000000000000,-1.60000000000000) -- (5.70000000000000,-1.60000000000000);
	\draw[line width=1.0pt] (8.50000000000000,-1.30000000000000) -- (5.50000000000000,-1.30000000000000);

	\draw[line width=1.0pt] (6.15899968200095,-1.85439949120153) -- (5,0);
	\draw[line width=1.0pt] (5,0) -- (6.15899968200095,1.85439949120153);

	\fill (-0.5+7,1.6) circle (2.4pt);
	\fill (0.5+7,1.6) circle (2.4pt);
	\fill (-0.5+7,1.3) circle (2.4pt);
	\fill (0.5+7,1.3) circle (2.4pt);

	\fill (-0.5+7,-1.6) circle (2.4pt);
	\fill (0.5+7,-1.6) circle (2.4pt);
	\fill (-0.5+7,-1.3) circle (2.4pt);
	\fill (0.5+7,-1.3) circle (2.4pt);

    \draw[line width=1.0pt,rotate around={0:(1,1)}] (4.45,-0.25) [partial ellipse=20:180:0.6cm and 0.6cm];
    \draw[line width=1.0pt,rotate around={0:(1,1)}] (4.45,0.2) [partial ellipse=180:340:0.6cm and 0.6cm];

    \draw[line width=1.0pt,rotate around={0:(1,1)}] (4.45+5.1,-0.25) [partial ellipse=0:160:0.6cm and 0.6cm];
    \draw[line width=1.0pt,rotate around={0:(1,1)}] (4.45+5.1,0.2) [partial ellipse=200:360:0.6cm and 0.6cm];

	\node at (11,0.5) {$\mathbb{Z}_2$};
	\node at (11,0) {$\longleftarrow$};

%-----------------------------------------------

	\draw[line width=1.0pt] (12.7000000000000,1.60000000000000) -- (15.3000000000000,1.60000000000000);
	\draw[line width=1.0pt] (12.5000000000000,1.30000000000000) -- (15.5000000000000,1.30000000000000);

	\draw[line width=1.0pt] (14.8410003179990,1.85439949120153) -- (16.1589996820010,-0.254399491201526);
	\draw[line width=1.0pt] (14.4398021025836,1.85364667090929) -- (15.9601978974164,-0.553646670909291);

	\draw[line width=1.0pt] (16.1589996820010,0.254399491201526) -- (14.8410003179990,-1.85439949120153);
	\draw[line width=1.0pt] (14.4398021025836,-1.85364667090929) -- (15.9601978974164,0.553646670909291);

	\draw[line width=1.0pt] (15.3000000000000,-1.60000000000000) -- (12.7000000000000,-1.60000000000000);
	\draw[line width=1.0pt] (15.5000000000000,-1.30000000000000) -- (12.5000000000000,-1.30000000000000);

	\draw[line width=1.0pt] (13.1589996820010,-1.85439949120153) -- (11.8410003179990,0.254399491201526);
	\draw[line width=1.0pt] (13.5601978974164,1.85364667090929) -- (12.0398021025836,-0.553646670909291);

	\draw[line width=1.0pt] (11.8410003179990,-0.254399491201526) -- (13.1589996820010,1.85439949120153);
	\draw[line width=1.0pt] (13.5601978974164,-1.85364667090929) -- (12.0398021025836,0.553646670909291);

\end{tikzpicture}
\caption{$\mathbb{Z}_2^2$-cover of the six $(-1)$-curves in $\Bl_3\mathbb{P}^2$. The points correspond to the branching points.}
\label{Z22coverand12(-2)-curves}
\end{figure}

\begin{figure}
\begin{tikzpicture}[scale=0.8]

    \draw [line width=1.0pt] (7.,3.)-- (4.401923788646684,1.5);
    \draw [line width=1.0pt] (4.401923788646684,1.5)-- (7.,2.);
    \draw [line width=1.0pt] (5.267949192431123,1.)-- (7.,2.);
    \draw [line width=1.0pt] (5.267949192431123,1.)-- (7.,3.);
    \draw [line width=1.0pt] (7.,2.)-- (8.732050807568877,1.);
    \draw [line width=1.0pt] (7.,3.)-- (9.598076211353316,1.5);
    \draw [line width=1.0pt] (8.732050807568877,1.)-- (7.,3.);
    \draw [line width=1.0pt] (7.,2.)-- (9.598076211353316,1.5);
    \draw [line width=1.0pt] (8.732050807568877,1.)-- (8.732050807568877,-1.);
    \draw [line width=1.0pt] (8.732050807568877,-1.)-- (9.598076211353316,1.5);
    \draw [line width=1.0pt] (9.598076211353316,1.5)-- (9.598076211353316,-1.5);
    \draw [line width=1.0pt] (9.598076211353316,-1.5)-- (8.732050807568877,1.);
    \draw [line width=1.0pt] (8.732050807568877,-1.)-- (7.,-2.);
    \draw [line width=1.0pt] (7.,-2.)-- (9.598076211353316,-1.5);
    \draw [line width=1.0pt] (8.732050807568877,-1.)-- (7.,-3.);
    \draw [line width=1.0pt] (7.,-3.)-- (9.598076211353316,-1.5);
    \draw [line width=1.0pt] (5.267949192431123,-1.)-- (7.,-2.);
    \draw [line width=1.0pt] (4.401923788646684,-1.5)-- (7.,-3.);
    \draw [line width=1.0pt] (7.,-3.)-- (5.267949192431123,-1.);
    \draw [line width=1.0pt] (7.,-2.)-- (4.401923788646684,-1.5);
    \draw [line width=1.0pt] (5.267949192431123,1.)-- (5.267949192431123,-1.);
    \draw [line width=1.0pt] (4.401923788646684,1.5)-- (4.401923788646684,-1.5);
    \draw [line width=1.0pt] (5.267949192431123,1.)-- (4.401923788646684,-1.5);
    \draw [line width=1.0pt] (5.267949192431123,-1.)-- (4.401923788646684,1.5);

    \draw [fill=black] (7.,2.) circle (2.4pt);
    \draw [fill=black] (7.,-2.) circle (2.4pt);
    \draw [fill=black] (5.267949192431123,1.) circle (2.4pt);
    \draw [fill=black] (8.732050807568877,1.) circle (2.4pt);
    \draw [fill=black] (8.732050807568877,-1.) circle (2.4pt);
    \draw [fill=black] (5.267949192431123,-1.) circle (2.4pt);
    \draw [fill=black] (7.,3.) circle (2.4pt);
    \draw [fill=black] (7.,-3.) circle (2.4pt);
    \draw [fill=black] (4.401923788646684,1.5) circle (2.4pt);
    \draw [fill=black] (9.598076211353316,1.5) circle (2.4pt);
    \draw [fill=black] (9.598076211353316,-1.5) circle (2.4pt);
    \draw [fill=black] (4.401923788646684,-1.5) circle (2.4pt);

    \node at (7,2-0.5) {$2$};
    \node at (7,-2+0.5) {$8$};
    \node at (5.267949192431123+0.5,1-0.3) {$12$};
    \node at (8.732050807568877-0.5,1-0.3) {$4$};
    \node at (8.732050807568877-0.5,-1+0.3) {$6$};
    \node at (5.267949192431123+0.5,-1+0.3) {$10$};
    \node at (7.,3.+0.4) {$1$};
    \node at (7,-3-0.4) {$7$};
    \node at (4.401923788646684-0.5,1.5+0.3) {$11$};
    \node at (9.598076211353316+0.4,1.5+0.3) {$3$};
    \node at (9.598076211353316+0.4,-1.5-0.3) {$5$};
    \node at (4.401923788646684-0.4,-1.5-0.3) {$9$};

\end{tikzpicture}
\caption{Dual graph of the $12$ $(-2)$-curves on a $D_{1,6}$-polarized Enriques surface.}
\label{D16-En-dualgraph12(-2)-curves}
\end{figure}
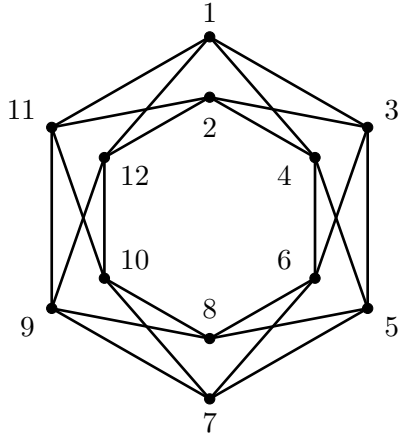

\begin{lemma}
Let $S$ be a $D_{1,6}$-polarized Enriques surface and let $S\rightarrow\Bl_3\mathbb{P}^2$ be the corresponding $\mathbb{Z}_2^2$-cover branched along $\sum_{i=1}^3(\ell_i+\ell_i')$. Let $E_i,E_i'\subseteq S$ be the preimages of $\ell_i,\ell_i'$ respectively. Then $E_i,E_i'$ are half-fibers. Additionally, we have the following numerical equivalences:
\begin{align*}
E_1&\equiv E_1'\equiv\frac{1}{2}(R_1+R_2+R_3+R_4)\equiv\frac{1}{2}(R_7+R_8+R_9+R_{10}),\\
E_2&\equiv E_2'\equiv\frac{1}{2}(R_3+R_4+R_5+R_6)\equiv\frac{1}{2}(R_9+R_{10}+R_{11}+R_{12}),\\
E_3&\equiv E_3'\equiv\frac{1}{2}(R_5+R_6+R_7+R_8)\equiv\frac{1}{2}(R_{11}+R_{12}+R_1+R_2).
\end{align*}
\end{lemma}

\begin{proof}
From the bi-double cover construction in the proof of Lemma~\ref{computation12nodesond16enriquessurface} we can see that $E_i,E_i'$ are genus one curves and that $E_i\cdot E_j=E_i\cdot E_j'=E_i'\cdot E_j'=1$ for $i\neq j$. This guarantees that $E_i,E_i'$ are half-fibers. The numerical equivalence can be understood as follows. $R_1+R_2+R_3+R_4$ is an arithmetic genus one curve which intersects $E_1$ giving zero. So $R_1+R_2+R_3+R_4\in|2E_1|=|2E_1'|$. The other equivalences are analogous.
\end{proof}

We now compute a $\mathbb{Z}$-basis for $\Num(S)$.

\begin{lemma}
\label{num-of-D16-Enriques}
Let $S$ be a $D_{1,6}$-polarized Enriques surface, and consider the smooth rational curves $R_1,\ldots,R_{12}$ as in Figure~\ref{D16-En-dualgraph12(-2)-curves}. Then a $\mathbb{Z}$-basis for $\Num(S)$ is given by
\begin{gather*}
R_1,~R_2,~R_3,~R_5,~R_7,~R_9,~\frac{1}{2}(R_1+R_3+R_5+R_7+R_9+R_{11}),\\
E_1\equiv\frac{1}{2}(R_1+R_2+R_3+R_4),~E_2\equiv\frac{1}{2}(R_3+R_4+R_5+R_6),~E_3\equiv\frac{1}{2}(R_5+R_6+R_7+R_8).
\end{gather*}
\end{lemma}

\begin{proof}
We follow the strategy of Remark~\ref{remark-how-to-compute-num} to determine a basis of $\Num(S)$. Let $L$ be the sublattice of $\Num(S)$ generated by the following elements:
\[
R_1,~R_2,~R_3,~R_5,~R_7,~R_9,~R_{11},~E_1,~E_2,~E_3.
\]
Let $B$ be the $10\times10$ matrix of intersection of the above generators of $L$. As the determinant of $B$ is nonzero, we have that the lattice $L$ has rank $10$. As $\Num(S)$ is an even overlattice of $L$, it corresponds to an isotropic subgroup of the discriminant group $L^*/L$, which we now compute. The rows of $B^{-1}$ generate $L^*$, and to better identify a set of generators of $L^*/L$ we compute the Smith normal form of $B^{-1}$. The function \texttt{smith\_form()} in SageMath returns two matrices $M_1,M_2\in\SL_{10}(\mathbb{Z})$ such that $M_1B^{-1}M_2$ is the diagonal matrix $\mathrm{diag}\left(1,\ldots,1,\frac{1}{2},\frac{1}{2}\right)$. This implies that $L^*/L\cong\mathbb{Z}_2^2$, and the rows of $M_1B^{-1}$ give an alternative basis for $L^*$. Using these we can find that the isotropic vectors of $L^*/L$ are the classes of:
\[
\frac{1}{2}(R_1+R_3+R_5+R_7+R_9+R_{11}),~\frac{1}{2}(R_2+R_3+R_5+R_7+R_9+R_{11}).
\]
Note that these cannot both be in $\Num(S)$, otherwise $\frac{1}{2}(R_1+R_2)$ would be an element of $\Num(S)$, which is impossible as it has odd square. Moreover, one of the two vectors above has to be in $\Num(S)$, so up to relabeling $R_1$ and $R_2$ we fix that $\frac{1}{2}(R_1+R_3+R_5+R_7+R_9+R_{11})\in\Num(S)$, and together with $L$ they generate $\Num(S)$. To obtain the claimed $\mathbb{Z}$-basis, we can then drop the curve $R_{11}$, which became redundant.
\end{proof}

\begin{proposition}
Let $S$ be $D_{1,6}$-polarized Enriques surface and let $\mathcal{R}$ be the configuration of $12$ smooth rational curves on $S$ as in Figure~\ref{D16-En-dualgraph12(-2)-curves}. The elliptic fibrations in $\mathcal{E}(S,\mathcal{R})$ are
\[
3 \times (2\widetilde{A}_3^{\f}), 24 \times (\widetilde{A}_3^{\hf}), 32 \times (\widetilde{A}_5^{\f}), 32 \times (\widetilde{A}_5^{\hf}), 12 \times (\widetilde{D}_4^{\f}), 24 \times (\widetilde{D}_5^{\f}), 48 \times (\widetilde{D}_6^{\f}).
\]

We have that $\cnd(S,\mathcal{R})=10$, and therefore $\nd(S)=10$. An explicit isotropic sequence realizing $\cnd(S,\mathcal{R})=10$ is given by the numerical equivalence classes of:
\begin{align*}
    E_1\equiv\frac{1}{2}(R_1+R_2+R_3+R_4) \qquad &(\widetilde{A}_3^{\f})\\
    E_2\equiv\frac{1}{2}(R_3+R_4+R_5+R_6) \qquad &(\widetilde{A}_3^{\f})\\
    E_3\equiv\frac{1}{2}(R_5+R_6+R_7+R_8) \qquad &(\widetilde{A}_3^{\f})\\
    \frac{1}{2}(R_{1} + R_{3} + R_{5} + R_{8} + R_{9} +  R_{12}) \qquad &(\widetilde{A}_5^{\f})\\
    \frac{1}{2}(R_{1} + R_{4} + R_{5} + R_{7} + R_{9} +  R_{12})  \qquad &(\widetilde{A}_5^{\f})\\
    \frac{1}{2}(R_{1} + R_{4} + R_{5} + R_{8} + R_{9} +  R_{11})  \qquad &(\widetilde{A}_5^{\f})\\
    \frac{1}{2}(R_{1} + R_{3} + R_{5} + R_{7} + R_{9} +  R_{11})  \qquad &(\widetilde{A}_5^{\f})\\
    \frac{1}{2}(2R_1 + R_3 + R_4 + R_{11} + R_{12}) \qquad &(\widetilde{D}_4^{\f})\\
    \frac{1}{2}( R_{3} + R_{4} + 2R_5 + R_7 + R_8) \qquad &(\widetilde{D}_4^{\f})\\
    \frac{1}{2}(R_{7} + R_{8} + 2R_9 + R_{11} + R_{12}) \qquad &(\widetilde{D}_4^{\f}).
\end{align*}
\end{proposition}

\begin{remark}
$\cnd(S,\mathcal{R})=10$ can be realized exactly in $16$ different ways, and these involve the same type of elliptic fibrations.  
\end{remark}

\begin{remark}
\label{D16-new-example}
Let $S$ be a $D_{1,6}$-polarized Enriques surface. $S$ is not general nodal because, for instance, the $(-2)$-curves $R_1,R_3$ are not equivalent modulo $2\Num(S)$: if by contradiction $R_1-R_3\in2\Num(S)$, then $(R_1-R_3)\cdot R_2$ should be even. However, $(R_1-R_3)\cdot R_2=1$. Moreover, a general $S$ does not have finite automorphism group because Enriques surfaces with finite automorphism group come at most in a one-dimensional family. However, we have a $4$-dimensional family of $D_{1,6}$-polarized Enriques surfaces. Alternatively, the automorphism group of a $D_{1,6}$-polarized Enriques surface is infinite because the dual graph of smooth rational curves in Figure~\ref{D16-En-dualgraph12(-2)-curves} is not a subgraph of the graphs in Figures~\ref{fig_dual_graph_KonI}--\ref{fig_dual_graph_K7}. These are the dual graphs of all smooth rational curves on Enriques surfaces with finite automorphism group, which are discussed in \S\,\ref{sec:ndKondoExamples}. Finally, a very general $S$ is not Hessian. To prove this, let $X\rightarrow S$ be the universal K3 covering. Then, by \cite[Theorem~4.6~(iii)]{Sch18} we know that the discriminant group of $\mathrm{NS}(X)$ is isomorphic to $\mathbb{Z}_2^2\oplus\mathbb{Z}_4^2$. On the other hand, the N\'eron--Severi group of the K3 cover of a Hessian Enriques surface has discriminant group isomorphic to $\mathbb{Z}_2^4\oplus\mathbb{Z}_3$ by \cite[\S\,4]{Kon12}.
\end{remark}

%---------------------------------------------------------------------------------

\section{Enriques surfaces with eight disjoint smooth rational curves}
\label{sec:LopesPardini}

In \cite{MLP02} Mendes Lopes and Pardini classified complex Enriques surfaces with eight disjoint smooth rational curves. These form two $2$-dimensional families, both obtained from a product of two elliptic curves, $A\coloneqq D_1\times D_2$, as the minimal resolution of a finite quotient of $A$.
We recall their constructions, which come with a distinguished configuration of smooth rational curves, and apply our code to these configurations.

%-------------------------------------------

\subsection{Example 1}
\label{ex-1-MLP}

Let $a\in D_1$ and $b\in D_2$ be $2$-torsion points, and let $e_1,e_2$ be generators for $\mathbb{Z}_2^2$. Let $e_1,e_2$ act on $A$ as follows:
\begin{align*}
e_1\cdot(x_1,x_2)&=(-x_1,x_2+b),\\
e_2\cdot(x_1,x_2)&=(x_1+a,-x_2).
\end{align*}
The quotient of $A$ by this $\mathbb{Z}_2^2$-action is a surface $\Sigma$ with eight $A_1$ singularities. Its minimal resolution $S$ is an Enriques surface whose universal cover $X$, a Kummer surface, is the resolution of $A/(e_1+e_2)$ at its 16 singular points. $S$ admits two elliptic fibrations induced by the projections $p_i\colon\Sigma \to D_i/\mathbb{Z}_2^2\cong\mathbb{P}^1$, $i=1,2$. Each $p_i$ has two double fibers $F_i,F_i'$ supported on two smooth rational curves. Four of the $A_1$ singularities lie on $F_i$, and the other four on $F_i'$. Moreover, each $F_1,F_1'$ intersects each $F_2,F_2'$ in exactly two $A_1$ singularities. Therefore, the elliptic fibration $f_i \colon S\to \Sigma\xrightarrow{p_i}\pr 1$ has two fibers of Kodaira type $\widetilde{D}_4$. The configuration of $12$ smooth rational curves $R_1,\ldots,R_{12}$ on $S$ which arises from the singular fibers of $f_1,f_2$ is pictured in Figure~\ref{fig_dual_graph_LP1}.

\begin{figure}
\begin{tikzpicture}[scale=0.8]

	\draw[line width=1.0pt] (0,2) -- (-1.6,1.6);
	\draw[line width=1.0pt] (0,2) -- (1.6,1.6);
	\draw[line width=1.0pt] (0,2) -- (-1,1);	
	\draw[line width=1.0pt] (0,2) -- (1,1);
	
	\draw[line width=1.0pt] (-2,0) -- (-1.6,1.6);
	\draw[line width=1.0pt] (2,0) -- (1.6,1.6);
	\draw[line width=1.0pt] (-2,0) -- (-1,1);	
	\draw[line width=1.0pt] (2,0) -- (1,1);
	
	\draw[line width=1.0pt] (0,-2) -- (-1.6,-1.6);
	\draw[line width=1.0pt] (0,-2) -- (1.6,-1.6);
	\draw[line width=1.0pt] (0,-2) -- (-1,-1);	
	\draw[line width=1.0pt] (0,-2) -- (1,-1);
	
	\draw[line width=1.0pt] (-2,0) -- (-1.6,-1.6);
	\draw[line width=1.0pt] (2,0) -- (1.6,-1.6);
	\draw[line width=1.0pt] (-2,0) -- (-1,-1);	
	\draw[line width=1.0pt] (2,0) -- (1,-1);

	\fill (0,2) circle (2.4pt);
	\fill (-2,0) circle (2.4pt);
	\fill (2,0) circle (2.4pt);
	\fill (0,-2) circle (2.4pt);

	\fill (-1.6,1.6) circle (2.4pt);
	\fill (-1,1) circle (2.4pt);
	\fill (1.6,1.6) circle (2.4pt);
	\fill (1,1) circle (2.4pt);
	\fill (-1.6,-1.6) circle (2.4pt);
	\fill (-1,-1) circle (2.4pt);
	\fill (1.6,-1.6) circle (2.4pt);
	\fill (1,-1) circle (2.4pt);

	\node at (0,2+0.4) {$1$};
	\node at (-2-0.4,0) {$10$};
	\node at (2+0.4,0) {$4$};
	\node at (0,-2-0.4) {$7$};
	\node at (-1.6-0.3,1.6+0.3) {$11$};
	\node at (-1+0.3,1-0.3) {$12$};
	\node at (1.6+0.3,1.6+0.3) {$2$};
	\node at (1-0.3,1-0.3) {$3$};
	\node at (-1.6-0.3,-1.6-0.3) {$8$};
	\node at (-1+0.3,-1+0.3) {$9$};
	\node at (1.6+0.3,-1.6-0.3) {$5$};
	\node at (1-0.3,-1+0.3) {$6$};

\end{tikzpicture}
\caption{Dual graph of the rational curves $R_1,\ldots,R_{12}$ in \cite[Example~1]{MLP02}.}
\label{fig_dual_graph_LP1}
\end{figure}
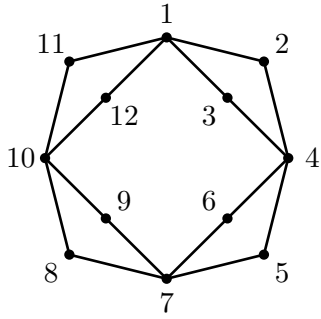

\begin{proposition}
\label{prop:num_LP1}
For an Enriques surface $S$ as above, let $R_1,\ldots,R_{12}$ be the $12$ smooth rational curves as in Figure~\ref{fig_dual_graph_LP1}. Then the lattice $\Num(S)$ is generated by
\begin{gather*}
R_1,~R_2,~R_3,~R_4,~R_5,~R_7,~R_9,\\
A=\frac{1}{2}(R_2+R_3+R_5+R_6),~B=\frac{1}{2}(R_2+R_3+R_{11}+R_{12}),\\
C=\frac{1}{2}(R_1+R_2+R_4+R_5+R_7+R_8+R_{10}+R_{11}).
\end{gather*}
\end{proposition}

\begin{proof}
By Lemma~\ref{possiblehalf-fibersincharacteristiczero}, the elliptic configurations with dual graph $\widetilde{D}_4$ are divisible by $2$ in $\Num(S)$. Hence, $A$ and $B$ are elements of $\Num(S)$.

Now consider the $\widetilde{A}_7$-type diagrams in Figure~\ref{fig_dual_graph_LP1} and assume by contradiction that they are all half-fibers. By Lemma~\ref{tooltodistinguishfiberfromhalf-fiber}, the preimages of $R_1+R_2+R_4+R_5+R_7+R_8+R_{10}+R_{11}$ and $R_1+R_3+R_4+R_5+R_7+R_8+R_{10}+R_{11}$ are connected in the covering K3, and this forces the preimage of $F_1=R_1+R_2+R_3+R_4$ to be disconnected, which means that $F_1$ is a fiber. On the other hand, also $F_2=R_2+R_3+2R_4+R_5+R_6$ is a fiber, which creates a contradiction as $F_1\cdot F_2=2$ is not divisible by $4$. This shows that there exists a curve of type $\widetilde{A}_7$ which is a fiber. Up to relabeling $R_2$ and $R_3$, we can fix that $R_1+R_2+R_4+R_5+R_7+R_8+R_{10}+R_{11}$ is a fiber.

Finally, we can conclude that the elements in $\Num(S)$ in the statement form a basis, since their intersection matrix has determinant $1$.
\end{proof}

\begin{proposition}
Let $S$ be an Enriques surface as in \S\,\ref{ex-1-MLP} and let $\mathcal{R}$ be the configuration of $12$ smooth rational curves on $S$ as in Figure~\ref{fig_dual_graph_LP1}. The elliptic fibrations in $\mathcal{E}(S,\mathcal{R})$ are
\[
2 \times (2 \widetilde{A}_3^{\hf}), 8 \times (\widetilde{A}_7^{\f}), 8 \times (\widetilde{A}_7^{\hf}), 2 \times (2\widetilde{D}_4^{\f}), 8 \times (\widetilde{D}_6^{\f}), 16 \times (\widetilde{D}_8^{\f}).
\]

We have that $\cnd(S,\mathcal{R})=8$, and therefore $\nd(S)\geq8$. An explicit isotropic sequence realizing $\cnd(S,\mathcal{R})=8$ is given by the numerical equivalence classes of:
\begin{align*}
R_{1}+R_{2}+R_{3}+R_{4} \qquad &(\widetilde{A}_3^{\hf})\\
\frac{1}{2}(R_{1}+R_{2}+R_{4}+R_{5}+R_{7}+R_{8}+R_{10}+R_{11}) \qquad &(\widetilde{A}_7^{\f})\\
\frac{1}{2}(R_{1}+R_{2}+R_{4}+R_{5}+R_{7}+R_{9}+R_{10}+R_{12}) \qquad &(\widetilde{A}_7^{\f})\\
\frac{1}{2}(R_{1}+R_{3}+R_{4}+R_{5}+R_{7}+R_{8}+R_{10}+R_{12}) \qquad &(\widetilde{A}_7^{\f})\\
\frac{1}{2}(R_{1}+R_{3}+R_{4}+R_{5}+R_{7}+R_{9}+R_{10}+R_{11}) \qquad &(\widetilde{A}_7^{\f})\\
\frac{1}{2}(R_{2}+R_{3}+2R_{4}+R_{5}+R_{6}) \qquad &(\widetilde{D}_4^{\f})\\
\frac{1}{2}(2R_{1}+R_{2}+R_{3}+R_{11}+R_{12}) \qquad &(\widetilde{D}_4^{\f})\\
\frac{1}{2}(R_{2}+R_{3}+2R_{4}+2R_{5}+2R_{7}+R_{8}+R_{9})\qquad &(\widetilde{D}_6^{\f}).
\end{align*}
\end{proposition}

\begin{remark}
$\cnd(S,\mathcal{R})=8$ can be realized exactly by $8$ different isotropic sequences, which all have the same type. There are three other types of $\sR$-saturated sequences in Figure~\ref{fig_dual_graph_LP1}:
\begin{itemize}
    \item $24$ sequences of length $7$ and type $(2\widetilde{A}_3^\hf), 4 \times (\widetilde{A}_7^\f), 2 \times (2\widetilde{D}_4^\f)$.
    \item $8$ sequences of length $5$ and type $4 \times (\widetilde{A}_7^\f), (\widetilde{A}_7^\hf)$.
    \item $32$ sequences of length $5$ and type $2 \times (\widetilde{A}_7^\f), (\widetilde{D}_4^\f), (\widetilde{D}_6^\f), (\widetilde{D}_8^\f)$.
\end{itemize}
\end{remark}

%-------------------------------------------

\subsection{Example 2} 
\label{ex-2-MLP}

Let $a_i\in D_1$ and $b_i\in D_2$, $i=1,2,3$, denote the points of order 2, and let $e_1,e_2,e_3$ be the standard generators for $\Z_2^3$. Let $\Z_2^3$ act on $A$ by
\begin{align*}
e_1\cdot(x_1,x_2)&=(x_1+a_1,x_2+b_1),\\
e_2\cdot(x_1,x_2)&=(x_1+a_2,-x_2),\\
e_3\cdot(x_1,x_2)&=(-x_1,x_2+b_3).
\end{align*}
Again, we denote by $\pi\colon A \to (D_1\times D_2)/\Z_2^3\eqqcolon \Sigma$ the quotient map. One shows that $\Sigma$ has eight $A_1$ singularities and its minimal resolution $S$ is an Enriques surface with eight disjoint smooth rational curves. The projections of $A$ onto the two factors descend to elliptic fibrations $f_i\colon S \to \Sigma\xrightarrow{p_i} \pr 1$. For $i=1,2$, $p_i$ has two double fibers $F_i,F_i'$, each passing through four $A_1$ singularities of $\Sigma$. $F_1$ intersects $F_2$ in the four $A_1$ singularities, and $F_2'$ in two smooth points of $\Sigma$. $F_1'$ intersects $F_2'$ in the four $A_1$ singularities, and $F_2$ in two smooth points of $\Sigma$. Therefore, each elliptic fibrations $f_i$ has two fibers of type $\widetilde{D}_4$. The dual graph of the rational curves $R_1,\ldots,R_{12}$ arising from the singular fibers of $f_1,f_2$ is depicted in Figure~\ref{fig_dual_graph_LP2}.

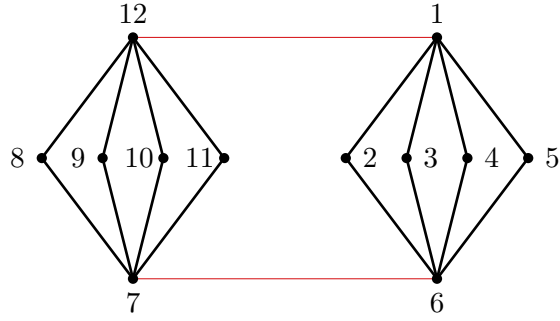
\begin{figure}
\begin{tikzpicture}[scale=0.8]

	\draw[color=ffqqqq] (-2.5,2) -- (2.5,2);
	\draw[color=ffqqqq] (-2.5,-2) -- (2.5,-2);

	\fill (-4,0) circle (2.4pt);
	\fill (-3,0) circle (2.4pt);
	\fill (-2,0) circle (2.4pt);
	\fill (-1,0) circle (2.4pt);
	\fill (1,0) circle (2.4pt);
	\fill (2,0) circle (2.4pt);
	\fill (3,0) circle (2.4pt);
	\fill (4,0) circle (2.4pt);
	\fill (-2.5,2) circle (2.4pt);
	\fill (-2.5,-2) circle (2.4pt);
	\fill (2.5,2) circle (2.4pt);
	\fill (2.5,-2) circle (2.4pt);

	\draw[line width=1.0pt] (-2.5,2) -- (-4,0);
	\draw[line width=1.0pt] (-2.5,2) -- (-3,0);
	\draw[line width=1.0pt] (-2.5,2) -- (-2,0);
	\draw[line width=1.0pt] (-2.5,2) -- (-1,0);
	\draw[line width=1.0pt] (-2.5,-2) -- (-1,0);
	\draw[line width=1.0pt] (-2.5,-2) -- (-2,0);
	\draw[line width=1.0pt] (-2.5,-2) -- (-3,0);
	\draw[line width=1.0pt] (-2.5,-2) -- (-4,0);

	\draw[line width=1.0pt] (2.5,2) -- (4,0);
	\draw[line width=1.0pt] (2.5,2) -- (3,0);
	\draw[line width=1.0pt] (2.5,2) -- (2,0);
	\draw[line width=1.0pt] (2.5,2) -- (1,0);
	\draw[line width=1.0pt] (2.5,-2) -- (1,0);
	\draw[line width=1.0pt] (2.5,-2) -- (2,0);
	\draw[line width=1.0pt] (2.5,-2) -- (3,0);
	\draw[line width=1.0pt] (2.5,-2) -- (4,0);

	\node at (-4-0.4,0) {$8$};
	\node at (-3-0.4,0) {$9$};
	\node at (-2-0.4,0) {$10$};
	\node at (-1-0.4,0) {$11$};
	\node at (1+0.4,0) {$2$};
	\node at (2+0.4,0) {$3$};
	\node at (3+0.4,0) {$4$};
	\node at (4+0.4,0) {$5$};
	\node at (-2.5,2+0.4) {$12$};
	\node at (-2.5,-2-0.4) {$7$};
	\node at (2.5,2+0.4) {$1$};
	\node at (2.5,-2-0.4) {$6$};

\end{tikzpicture}
\caption{Dual graph of the rational curves $R_1,\ldots,R_{12}$ in \cite[Example~2]{MLP02}. The colored edges joining the vetices $1,12$ and $6,7$ indicate intersection 2 between the corresponding curves.}
\label{fig_dual_graph_LP2}
\end{figure}

\begin{proposition}
\label{prop:num_LP2}
For an Enriques surface $S$ as above, let $R_1,\ldots,R_{12}$ be the $12$ smooth rational curves as in Figure~\ref{fig_dual_graph_LP2}. Then the lattice $\Num(S)$ is generated by
\begin{gather*}
R_1,~R_2,~R_3,~R_4,~R_7,~R_8,\\
A=\frac{1}{2}(R_2+R_3+R_4+R_5),~B=\frac{1}{2}(R_1+R_2+R_3+R_6),~C=\frac{1}{2}(R_7+R_8+R_9+R_{12}),\\
\frac{1}{2}(R_1+R_2+R_5+R_8+R_{10})+\frac{1}{4}(R_2+R_3+R_4+R_5).
\end{gather*}
\end{proposition}

\begin{proof}
The elliptic configurations of type $\widetilde{D}_4$ on $S$ guarantee that $R_2+R_3+R_4+R_5$ and $R_8+R_9+R_{10}+R_{11}$ are elements of $2\Num(S)$. We can determine more elliptic configurations in $2\Num(S)$ as follows. Consider the elliptic configurations of type $\widetilde{A}_3$ on the right-hand side of Figure~\ref{fig_dual_graph_LP2}, and assume by contradiction that these are all half-fibers. Then, by Lemma~\ref{tooltodistinguishfiberfromhalf-fiber}, the preimages of $R_1+R_2+R_3+R_6$ and $R_1+R_2+R_4+R_6$ are connected. This forces the preimage of $R_1+R_3+R_4+R_6$ to be disconnected, which is a contradiction. As there exist elliptic configurations of type $\widetilde{A}_3$ on the right-hand side of Figure~\ref{fig_dual_graph_LP2}, we can assume up to relabeling that $R_1+R_2+R_3+R_6\in2\Num(S)$. An analogous argument for the elliptic configurations of type $\widetilde{A}_3$ on the left-hand side of Figure~\ref{fig_dual_graph_LP2} yields $R_7+R_8+R_9+R_{12}\in2\Num(S)$.

Now, define $L\subseteq\Num(S)$ to be the rank $10$ sublattice with basis given by
\[
R_1,~R_2,~R_3,~R_4,~R_7,~R_8,~R_{10},~A,~B,~C.
\]
The discriminant group of $L$ is $\mathbb{Z}_2^2$, so $L\subsetneq\Num(S)$ and we look for an element in $\Num(S)\setminus L$ by studying the isotropic elements in $L^*/L$. Using the same strategy as in the proof of Lemma~\ref{num-of-D16-Enriques}, we find that the isotropic vectors in $L^*/L$ are the classes of
\begin{align*}
\frac{1}{2}(R_1+R_2+R_5+R_8+R_{10})&+\frac{1}{4}(R_2+R_3+R_4+R_5),\\
\frac{1}{2}(R_1+R_3+R_5+R_8+R_{10})&+\frac{1}{4}(R_2+R_3+R_4+R_5).
\end{align*}
These cannot simultaneously be in $\Num(S)$, but one of them must be. So, up to relabeling $R_2,R_3$ we fix that the first one is in $\Num(S)$. Adding this vector to the generating set of $L$ and dropping $R_{10}$, which is now redundant, gives the claimed basis.
\end{proof}

\begin{proposition}
Let $S$ be an Enriques surface as in \S\,\ref{ex-2-MLP} and let $\mathcal{R}$ be the configuration of $12$ smooth rational curves on $S$ as in Figure~\ref{fig_dual_graph_LP2}. The elliptic fibrations in $\mathcal{E}(S,\mathcal{R})$ are
\[
1 \times (2\widetilde{A}_1^{\hf}), 4 \times (\widetilde{A}_3^{\f}), 8 \times (\widetilde{A}_3^{\hf}), 2 \times (2\widetilde{D}_4^{\f}).
\]

We have that $\cnd(S,\mathcal{R})=5$, and therefore $\nd(S)\geq5$. An explicit isotropic sequence realizing $\cnd(S,\mathcal{R})=5$ is given by the numerical equivalence classes of:
\begin{align*}
\frac{1}{2}(R_{1}+R_{2}+R_{3}+R_{6}) \qquad &(\widetilde{A}_3^{\f}) \ \ \qquad \ \
\frac{1}{2}(R_{1}+R_{4}+R_{5}+R_{6}) \qquad &(\widetilde{A}_3^{\f})\\
\frac{1}{2}(R_{7}+R_{8}+R_{9}+R_{12}) \qquad &(\widetilde{A}_3^{\f}) \ \ \qquad \ \
\frac{1}{2}(R_{7}+R_{10}+R_{11}+R_{12}) \qquad &(\widetilde{A}_3^{\f})\\
R_{1}+R_{12} \qquad &(\widetilde{A}_1^{\hf}).
\end{align*}
\end{proposition}

\begin{remark}
The only other isotropic sequences realizing $\cnd(S,\mathcal{R})=5$ are obtained by replacing $R_{1}+R_{12}$ with either $\frac{1}{2}(2R_1+R_2+R_3+R_4+R_5)$ or $\frac{1}{2}(2R_6+R_2+R_3+R_4+R_5)$. There is another type of $\mathcal{R}$-saturated sequences which has length $3$ and has type $2 \times (\widetilde{A}_3^\f), (\widetilde{A}_3^\hf)$.
\end{remark}

%---------------------------------------------------------------------------------

\section{Enriques surfaces with finite automorphism group, revisited}
\label{sec:ndKondoExamples}

In this section, we revisit the Enriques surfaces with finite automorphism group. These were classified in \cite{Kon86} into seven types, and their non-degeneracy invariants were computed in \cite{DK22}. Our code re-computes these non-degeneracy invariants and provides additional geometric information as outlined in the introduction. We work over $\mathbb{C}$. For the realizability of these examples in positive characteristic we refer to the discussion in \cite{Mar19}.

We will not review the constructions of Kond\=o's examples because we only need the (finite) dual graphs of all smooth rational curves $\sR$ one these surfaces. We recall these graphs in \S\,\ref{sec:config-smooth-rat-curves}. For each Enriques surface $S$ with finite automorphism group, we provide a basis $\mathcal{B}$ of $\Num(S)$ using $\mathbb{Q}$-linear combinations of elements in $\mathcal{R}$. Afterwards, we run our computer code with $\mathcal{R}$ and $\mathcal{B}$ to compute $\cnd(S,\mathcal{R})$, $\mathcal{E}(S,\mathcal{R})$, and the $\mathcal{R}$-saturated sequences. As all the half-fibers are supported on $\mathcal{R}$ by \cite{Kon86}, this recovers $\nd(S)$ and all the elliptic fibrations, and computes the saturated sequences.

%-------------------------------------------

\subsection{Bases for the lattices \texorpdfstring{$\Num(S)$}{Lg}}

\begin{lemma}
Let $S$ be an Enriques surface with finite automorphism group and consider the configuration of smooth rational curves on $S$ in the corresponding figure in \S\,\ref{sec:config-smooth-rat-curves}. Then, for each type, a basis for $\Num(S)$ is given by the numerical classes of the curves in Table~\ref{tbl:all-bases-Num-Enriques-finite-aut}.
\end{lemma}

\begin{proof}
We first need to verify that for each type, the $\mathbb{Q}$-cycles listed in the second column of Table~\ref{tbl:all-bases-Num-Enriques-finite-aut} are actual elements of $\Num(S)$. This is immediate for type VI. In type I, we have that $A,B,C$ are elements of $\Num(S)$ because $2A,2B,2C$ are elliptic configurations with dual graphs $\widetilde{D}_8,\widetilde{D}_8,\widetilde{E}_7$ respectively, which cannot be half-fibers by Lemma~\ref{possiblehalf-fibersincharacteristiczero}. A similar argument applies in type V. In type IV we have that $R_3+R_4+R_{13}+R_{16}+R_{19}$ is a fiber by \cite[Proposition~8.9.16]{DK22}. In type VII, all the elliptic configurations with dual graph $\widetilde{A}_4$ are fibers by \cite[Proposition~8.9.28]{DK22}, so $R_1+R_2+R_3+R_4+R_{15}$ is divisible by $2$ in $\Num(S)$.

For type II, $2A$ is an elliptic configuration with dual graph $\widetilde{D}_5$, hence $A\in\Num(S)$. Consider the arrangements of nine curves among $R_1,\ldots,R_{12}$ whose dual graph is $\widetilde{A}_8$. By \cite[Proposition~8.9.9]{DK22} we know that among these eight possible configurations, four are fibers and the other four are half-fibers. So, up to relabeling, we can assume that
\[R_1+R_2+R_3+R_5+R_6+R_7+R_9+R_{10}+R_{11}\]
is a fiber, hence $B\in\Num(S)$.

For type III, the subgraph $\Gamma$ induced by the vertices $R_1, \ldots, R_{12}$ is isomorphic to the graph in  Figure~\ref{fig_dual_graph_LP1}. We fix the following bijection between the curves in Figure~\ref{fig_dual_graph_LP1} and Figure~\ref{fig_dual_graph_K3}:
\begin{center}
\begin{tabular}{|c|c|c|c|c|c|c|c|c|c|c|c|c|}
\hline
\cite{DK22}, Type III (Figure~\ref{fig_dual_graph_K3}) &                        1 & 2 & 3 & 4 & 5 & 6 & 7 & 8 & 9 & 10 & 11 & 12 \\
\hline
\cite{MLP02}, Example I (Figure~\ref{fig_dual_graph_LP1}) & 1 & 2 & 4 & 6 & 7 & 9 & 10 & 11 & 3 & 12 & 5 & 8 \\
\hline
\end{tabular}
\end{center}
Moreover, the group of symmetries of the diagram in Figure~\ref{fig_dual_graph_K3} is isomorphic to that of $\Gamma$ \cite[\S\,8.9]{DK22}, and the transposition $(R_2 \ \ R_9)$ on $\Gamma$ corresponds to the product of transpositions
\[
\sigma\coloneqq(R_2 \ \ R_9)(R_{15} \ \ R_{16})(R_{19} \ \ R_{20}).
\]
Then, the same argument as that of Proposition~\ref{prop:num_LP1} applies: the only subtlety is the choice of $C$ up to a relabeling of $R_2$ and $R_9$, which corresponds to a choice between $C$ and $\sigma(C)$. Since $\sigma$ does not affect any other element in Table~\ref{tbl:all-bases-Num-Enriques-finite-aut}, type III, we can choose $C\in 2\Num(S)$.

To conclude the proof, it is enough to check for each type that the determinant of the $10\times10$ intersection matrix associated with the corresponding $10$ curves is equal to $\pm1$.
\end{proof}

\begin{table}
\centering\renewcommand\cellalign{lc}
\setcellgapes{3pt}\makegapedcells
\caption{Bases of $\Num(S)$ for each type of Enriques surface with finite automorphism group. The labeling of the curves refers to that of the figures in \S\,\ref{sec:config-smooth-rat-curves}.}
\label{tbl:all-bases-Num-Enriques-finite-aut}
\begin{tabular}{|c|c|}
\hline
Type & Basis of $\Num(S)$
\\
\hline
I
&
\makecell{
$R_1,~R_2,~R_3,~R_4,~R_5,~R_6,~R_7,$
\\[1ex]
$A=\frac{1}{2}(2R_1+R_2+R_4+2R_5+2R_6+2R_7+2R_8+R_9+R_{12}),$
\\[1ex]
$B=\frac{1}{2}(2R_1+2R_2+2R_3+2R_4+2R_5+R_6+R_8+R_9+R_{12}),$
\\[1ex]
$C=\frac{1}{2}(4R_1+3R_2+2R_3+R_4+R_6+2R_7+3R_8+2R_9).$
}
\\
\hline
II
&
\makecell{
$R_1,~R_2,~R_3,~R_4,~R_5,~R_7,~R_9,~R_{10},$
\\[1ex]
$A=\frac{1}{2}(R_2+2R_3+R_4+2R_5+R_6+R_8),$
\\[1ex]
$B=\frac{1}{2}(R_1+R_2+R_3+R_5+R_6+R_7+R_9+R_{10}+R_{11}).$
}
\\
\hline
III
&
\makecell{
$R_1,~R_2,~R_3,~R_5,~R_6,~R_9,~R_{11},$
\\[1ex]
$A=\frac{1}{2}(R_2+R_{4}+R_9+R_{11}),~B=\frac{1}{2}(R_2+R_8+R_9+R_{10}),$
\\[1ex]
$C=\frac{1}{2}(R_1+R_2+R_3+R_5+R_7+R_8+R_{11}+R_{12}).$
}
\\
\hline
IV
&
\makecell{
$R_1,~R_2,~R_3,~R_5,~R_6,~R_9,~R_{11},~R_{13},~R_{19},$
\\[1ex]
$\frac{1}{2}(R_3+R_4+R_{13}+R_{16}+R_{19}).$
}
\\
\hline
V
&
\makecell{
$R_1,~R_2,~R_3,~R_4,~R_5,~R_7,~R_{9},~R_{17},$
\\[1ex]
$A=\frac{1}{2}(2R_1+R_2+2R_4+4R_5+3R_6+3R_7+2R_8+R_9),$
\\[1ex]
$B=\frac{1}{2}(R_1+R_3+2R_4+3R_5+2R_6+2R_7+R_8).$
}
\\
\hline
VI
&
\makecell{
$R_1,~R_2,~R_3,~R_4,~R_5,~R_7,~R_{11},~R_{12},~R_{14},~R_{17}.$
}
\\
\hline
VII
&
\makecell{
$R_1,~R_2,~R_3,~R_4,~R_5,~R_6,~R_7,~R_9,~R_{11},$
\\[1ex]
$\frac{1}{2}(R_1+R_2+R_3+R_4+R_{15}).$
}
\\
\hline
\end{tabular}
\end{table}

%-------------------------------------------

\subsection{Output of the code: isotropic sequences}
\label{sec:info-from-the-program}

The next proposition follows by running our code with $\mathcal{R}$ and the bases of $\Num(S)$ given in Table~\ref{tbl:all-bases-Num-Enriques-finite-aut}.

\begin{proposition}
Let $S$ be the Enriques surface with finite automorphism group. Then, for each type, Table~\ref{tbl:isotropic-sequences-realizing-cnd(S,R)} gives an isotropic sequence realizing $\nd(S)$, together with the number of non-degenerate isotropic sequences of length $\nd(S)$. For the labeling of the curves, we refer to the figures in \S\,\ref{sec:config-smooth-rat-curves}.
\end{proposition}

\begin{table}
\centering\renewcommand\cellalign{lc}
\setcellgapes{3pt}\makegapedcells
\caption{Examples of isotropic sequences realizing $\nd(S)$ for the Enriques surfaces with finite automorphism group. The third column reports the number of non-degenerate isotropic sequences of length $\nd(S)$. For each isotropic class $[C]$, in bold we give the dual graph of the elliptic configuration $C$ or $2C$.}
\label{tbl:isotropic-sequences-realizing-cnd(S,R)}
\begin{tabular}{|c|c|c|c|}
\hline
Type
&
$\nd$
&
$\#$
&
Example of isotropic sequence
\\
\hline
I
&
4
&
2
&
\makecell{
$\frac{1}{2}(R_{10}+R_{11})~\pmb{ (\widetilde{A}_1^\f)},~R_9+R_{10}~\pmb{(\widetilde{A}_1^\hf)},$
\\[1ex]
$\frac{1}{2}(R_1+R_{2}+R_{4}+R_{5}+R_{6}+R_{7}+R_{8}+R_{9}+R_{12})~\pmb{(\widetilde{D}_8^\f)},$
\\[1ex]
$\frac{1}{2}(R_1+R_{2}+R_{3}+R_{4}+R_{5}+R_{6}+R_{8}+R_{9}+R_{12})~\pmb{(\widetilde{D}_8^\f)}.$
}
\\
\hline
II
&
7
&
1
&
\makecell{
$R_1+R_2+R_3+R_4~\pmb{(\widetilde{A}_3^\hf)},~R_5+R_6+R_7+R_8~\pmb{(\widetilde{A}_3^\hf)},$
\\[1ex]
$R_9+R_{10}+R_{11}+R_{12}~\pmb{(\widetilde{A}_3^\hf)},$
\\[1ex]
$\frac{1}{2}(R_1+R_2+R_3+R_5+R_6+R_7+R_9+R_{10}+R_{11})~\pmb{(\widetilde{A}_8^\f)},$
\\[1ex]
$\frac{1}{2}(R_1+R_2+R_3+R_5+R_7+R_8+R_9+R_{11}+R_{12})~\pmb{(\widetilde{A}_8^\f)},$
\\[1ex]
$\frac{1}{2}(R_1+R_3+R_4+R_5+R_6+R_7+R_9+R_{11}+R_{12})~\pmb{(\widetilde{A}_8^\f)},$
\\[1ex]
$\frac{1}{2}(R_1+R_3+R_4+R_5+R_7+R_8+R_9+R_{10}+R_{11})~\pmb{(\widetilde{A}_8^\f)}.$
}
\\
\hline
III
&
8
&
8
&
\makecell{
$R_1+R_2+R_3+R_9~\pmb{(\widetilde{A}_3^\hf)},\frac{1}{2}(R_2 + 2R_3 + R_4 + R_9 + R_{11})~\pmb{(\widetilde{D}_4^\f)},$
\\[1ex]
$\frac{1}{2}(2R_1 + R_2 + R_8 + R_9 + R_{10})~\pmb{(\widetilde{D}_4^\f)},$
\\[1ex]
$\frac{1}{2}(R_1 + R_2+ R_3+ R_4+ R_5+ R_6+ R_7+ R_8)~\pmb{(\widetilde{A}_7^\f)},$
\\[1ex]
$\frac{1}{2}(R_1 + R_2 + R_3 + R_4 + R_5 + R_7 + R_{10} + R_{12})~\pmb{(\widetilde{A}_7^\f)},$
\\[1ex]
$\frac{1}{2}(R_1 + R_3 + R_4 + R_5 + R_7 + R_8 + R_9 + R_{12})~\pmb{(\widetilde{A}_7^\f)},$
\\[1ex]
$\frac{1}{2}(R_1 + R_3 + R_4 + R_5 + R_6 + R_7 + R_9 + R_{10})~\pmb{(\widetilde{A}_7^\f)},$
\\[1ex]
$\frac{1}{2}(R_2 + 2R_3 + 2R_4 + 2R_5 + R_6 + R_9 + R_{12})~\pmb{(\widetilde{D}_6^\f)}.$
}
\\
\hline
IV
&
10
&
16
&
\makecell{
$R_1+R_{11}~\pmb{(\widetilde{A}_1^\hf)},~R_2+R_{10}~\pmb{(\widetilde{A}_1^\hf)},~R_5+R_{15}~\pmb{(\widetilde{A}_1^\hf)},$
\\[1ex]
$R_6+R_{13}~\pmb{(\widetilde{A}_1^\hf)},~R_{17}+R_{19}~\pmb{(\widetilde{A}_1^\hf)},~\frac{1}{2}(R_1 + R_2+ R_{14}+ R_{15}+ R_{19})~\pmb{(\widetilde{A}_4^\f)},$
\\[1ex]
$\frac{1}{2}(R_1 + R_2+ R_{13}+ R_{15}+ R_{20})~\pmb{(\widetilde{A}_4^\f)},~\frac{1}{2}(R_1 + R_2+ R_{13}+ R_{16}+ R_{19})~\pmb{(\widetilde{A}_4^\f)},$
\\[1ex]
$\frac{1}{2}(R_1 + R_4+ R_{13}+ R_{15}+ R_{19})~\pmb{(\widetilde{A}_4^\f)},~\frac{1}{2}(R_2 + R_3+ R_{13}+ R_{15}+ R_{19})~\pmb{(\widetilde{A}_4^\f)}.$
}
\\
\hline
V
&
7
&
20
&
\makecell{
$\frac{1}{2}(R_{11}+R_{12})~\pmb{(\widetilde{A}_1^\f)},~\frac{1}{2}(R_{11}+R_{13})~\pmb{(\widetilde{A}_1^\f)},~\frac{1}{2}(R_{11}+R_{20})~\pmb{(\widetilde{A}_1^\f)},$
\\[1ex]
$\frac{1}{2}(R_{14}+R_{17})~\pmb{(\widetilde{A}_1^\f)},~\frac{1}{2}(R_{15}+R_{18})~\pmb{(\widetilde{A}_1^\f)},~\frac{1}{2}(R_{16}+R_{19})~\pmb{(\widetilde{A}_1^\f)},$
\\[1ex]
$R_{8}+R_{11}~\pmb{(\widetilde{A}_1^\hf)}.$
}
\\
\hline
VI
&
10
&
1
&
\makecell{
$R_{1}+R_{20}~\pmb{(\widetilde{A}_1^\hf)},~R_{2}+R_{12}~\pmb{(\widetilde{A}_1^\hf)},~R_{3}+R_{17}~\pmb{(\widetilde{A}_1^\hf)},$
\\[1ex]
$R_{4}+R_{18}~\pmb{(\widetilde{A}_1^\hf)},~R_{5}+R_{13}~\pmb{(\widetilde{A}_1^\hf)},~R_{6}+R_{19}~\pmb{(\widetilde{A}_1^\hf)},$
\\[1ex]
$R_{7}+R_{14}~\pmb{(\widetilde{A}_1^\hf)},~R_{8}+R_{11}~\pmb{(\widetilde{A}_1^\hf)},~R_{9}+R_{15}~\pmb{(\widetilde{A}_1^\hf)},~R_{10}+R_{16}~\pmb{(\widetilde{A}_1^\hf)}.$
}
\\
\hline
VII
&
10
&
5
&
\makecell{
$\frac{1}{2}(R_{16}+R_{17})~\pmb{(\widetilde{A}_1^\f)},~\frac{1}{2}(R_{16}+R_{18})~\pmb{(\widetilde{A}_1^\f)},$
\\[1ex]
$\frac{1}{2}(R_{16}+R_{19})~\pmb{(\widetilde{A}_1^\f)},~\frac{1}{2}(R_{16}+R_{20})~\pmb{(\widetilde{A}_1^\f)},$
\\[1ex]
$\frac{1}{2}(R_{1}+R_{2}+R_3+R_4+R_{15})~\pmb{(\widetilde{A}_4^\f)},~\frac{1}{2}(R_{1}+R_{2}+R_9+R_{10}+R_{12})~\pmb{(\widetilde{A}_4^\f)},$
\\[1ex]
$\frac{1}{2}(R_{1}+R_{7}+R_8+R_9+R_{14})~\pmb{(\widetilde{A}_4^\f)},~\frac{1}{2}(R_{1}+R_{5}+R_6+R_{14}+R_{15})~\pmb{(\widetilde{A}_4^\f)},$
\\[1ex]
$\frac{1}{2}(R_{2}+R_{3}+R_7+R_{13}+R_{14})~\pmb{(\widetilde{A}_4^\f)},~\frac{1}{2}(R_{2}+R_{6}+R_{10}+R_{11}+R_{14})~\pmb{(\widetilde{A}_4^\f)}.$
}
\\
\hline
\end{tabular}
\end{table}

%-------------------------------------------

\subsection{Geometric considerations from the output data}
\label{sec:data-on-elliptic-fibrations}

%----------------------

We report some geometric considerations based on the data output of the code. This complements the data of \cite[\S\,8.9]{DK22}. In particular, the saturated sequences of each example are collected in Tables~\ref{Tab:sat-sequences-on-Kondo-I}, \ref{Tab:sat-sequences-on-Kondo-II}, \ref{Tab:sat-sequences-on-Kondo-III}, \ref{Tab:sat-sequences-on-Kondo-IV}, \ref{Tab:sat-sequences-on-Kondo-V}, \ref{Tab:sat-sequences-on-Kondo-VI}, and \ref{Tab:sat-sequences-on-Kondo-VII}.

\subsubsection{Type I}
The Enriques surface $S$ has the following elliptic fibrations (this agrees with \cite[Proposition~8.9.6]{DK22}):
\[
1 \times (\widetilde{A}_1^{\f} + \widetilde{A}_7^{\hf}),
2 \times (\widetilde{A}_1^{\hf} + \widetilde{E}_7^{\f}),
2 \times (\widetilde{D}_8^{\f}),
4 \times (\widetilde{E}_8^{\f}).
\]
The unique fibration of type $\widetilde{A}_1^{\f} + \widetilde{A}_7^{\hf}$ is $(1/2(R_{10}+R_{11}),R_1+R_2+R_3+R_4+R_5+R_6+R_7+R_8)$. The two fibrations of type $\widetilde{A}_1^{\hf} + \widetilde{E}_7^{\f}$ are $(R_9+R_{10},1/2(R_2+2R_3+3R_4+4R_5+3R_6+2R_7+1R_8+2R_{12}))$ and $(R_{11}+R_{12},1/2(4R_1+3R_2+2R_3+1R_4+R_6+2R_7+3R_8+2R_9))$.

\begin{table}[h!]
\centering\renewcommand\cellalign{lc}
\setcellgapes{3pt}\makegapedcells
\caption{Saturated sequences on the Enriques surface \cite[(3.1)~Example~I]{Kon86}.}
\label{Tab:sat-sequences-on-Kondo-I}
\begin{tabular}{|c|c|c|}
\hline
Length
&
Fibrations in the sequence
&
Cardinality
\\
\hline
$4$
&
$(\widetilde{A}_1^{\f} + \widetilde{A}_7^{\hf}), (\widetilde{A}_1^{\hf} + \widetilde{E}_7^{\f}), 2 \times (\widetilde{D}_8^{\f})$
&
$2$
\\
\hline
$3$
&
$(\widetilde{A}_1^{\hf} + \widetilde{E}_7^{\f}), (\widetilde{D}_8^{\f}), (\widetilde{E}_8^{\f})$
&
$4$
\\
\hline
\end{tabular}
\end{table}

%----------------------

\subsubsection{Type II}

We first recover that $S$ has the following elliptic fibrations, agreeing with \cite[Proposition~8.9.9]{DK22}:
\[
4 \times (\widetilde{A}_8^{\hf}),
4 \times (\widetilde{A}_8^{\f}),
6 \times (\widetilde{D}_8^{\f}),
3 \times (\widetilde{A}_3^{\hf}+\widetilde{D}_5^{\f}).
\]
The three fibrations of type $\widetilde{A}_3^{\hf}+\widetilde{D}_5^{\f}$ are $(R_1+R_2+R_3+R_4,1/2(R_6+2R_7+R_8+2R_9+R_{10}+R_{12}))$, $(R_5+R_6+R_7+R_8,1/2(2R_1+R_2+R_4+R_{10}+2R_{11}+R_{12}))$ and $(R_9+R_{10}+R_{11}+R_{12},1/2(R_2+2R_3+R_4+2R_{5}+R_{6}+R_{8}))$.

\begin{table}[h!]
\centering\renewcommand\cellalign{lc}
\setcellgapes{3pt}\makegapedcells
\caption{Saturated sequences on the Enriques surface \cite[(3.2)~Example~II]{Kon86}.}
\label{Tab:sat-sequences-on-Kondo-II}
\begin{tabular}{|c|c|c|}
\hline
Length
&
Fibrations in the sequence
&
Cardinality
\\
\hline
$7$
&
$3 \times (\widetilde{A}_3^{\hf} + \widetilde{D}_5^{\f}), 4 \times (\widetilde{A}_8^{\f})$
&
$1$
\\
\hline
$5$
&
$2 \times (\widetilde{A}_3^{\hf} + \widetilde{D}_5^{\f}), 2 \times (\widetilde{A}_8^{\f}), (\widetilde{D}_8^{\f})$
&
$6$
\\
\hline
$4$
&
$3 \times (\widetilde{A}_8^{\f}), (\widetilde{A}_8^{\hf})$
&
$4$
\\
\hline
\end{tabular}
\end{table}

%----------------------

\subsubsection{Type III}

As computed in \cite[Proposition~8.9.13]{DK22}, $S$ has the following elliptic fibrations:
\[
8 \times (\widetilde{A}_1^{\hf}+\widetilde{A}_7^{\hf}),
8 \times (\widetilde{A}_1^{\f}+\widetilde{A}_7^{\f}),
16 \times (\widetilde{D}_8^{\f}),
2 \times (2\widetilde{D}_4^{\f}),
8 \times (2\widetilde{A}_1^{\hf}+\widetilde{D}_6^{\f}),
2 \times (2\widetilde{A}_1^{\f}+2\widetilde{A}_3^{\hf}). 
\]
The two fibrations of type $(2\widetilde{A}_1^{\f}+2\widetilde{A}_3^{\hf})$ are $(1/2(R_{13}+R_{17}),1/2(R_{14}+R_{18}),R_{1}+R_{2}+R_{3}+R_{9},R_{5}+R_{6}+R_{7}+R_{12})$ and
$(1/2(R_{15}+R_{19}),1/2(R_{16}+R_{20}),R_{1}+R_{7}+R_{8}+R_{10},R_{3}+R_{4}+R_{5}+R_{11})$. 
The eight fibrations of type $(\widetilde{A}_1^{\f}+\widetilde{A}_7^{\f})$ and the eight fibrations of type $(\widetilde{A}_1^{\hf}+\widetilde{A}_7^{\hf})$ are given by a choice of one of the blue edges in Figure~\ref{fig_dual_graph_K3}, together with a suitable $\widetilde{A}_7^{\f}$ or $\widetilde{A}_7^{\hf}$. 
The eight fibrations of type $(2\widetilde{A}_1^{\hf}+\widetilde{D}_6^{\f})$ are given by the following pairs of $\widetilde{A}_1^{\hf}$ together with a suitable $\widetilde{D}_6^{\f}$:
$(R_{2}+R_{15},R_{9}+R_{16})$, $(R_{2}+R_{20},R_{9}+R_{19})$, $(R_{4}+R_{17},R_{11}+R_{14})$, $(R_{4}+R_{18},R_{11}+R_{13})$, $(R_{6}+R_{19},R_{12}+R_{16})$, $(R_{6}+R_{20},R_{12}+R_{15})$, $(R_{8}+R_{13},R_{10}+R_{14})$, $(R_{8}+R_{18},R_{10}+R_{17})$.

\begin{table}[h!]
\centering\renewcommand\cellalign{lc}
\setcellgapes{3pt}\makegapedcells
\caption{Saturated sequences on the Enriques surface \cite[(3.3)~Example~III]{Kon86}.}
\label{Tab:sat-sequences-on-Kondo-III}
\begin{tabular}{|c|c|c|}
\hline
Length
&
Fibrations in the sequence
&
Cardinality
\\
\hline
$8$
&
$4\times(\widetilde{A}_1^{\f} + \widetilde{A}_7^{\f}), (2 \widetilde{A}_1^{\f} + 2 \widetilde{A}_3^{\hf}), (2 \widetilde{A}_1^{\hf} + \widetilde{D}_6^{\f}), 2\times(2 \widetilde{D}_4^{\f})$
&
$8$
\\
\hline
$7$
&
$4 \times (\widetilde{A}_1^{\f} + \widetilde{A}_7^{\f}),  (2 \widetilde{A}_1^{\f} + 2 \widetilde{A}_3^{\hf}),  2 \times (2 \widetilde{D}_4^{\f})$
&
$24$
\\
\hline
$5$
&
$4 \times (\widetilde{A}_1^{\f} + \widetilde{A}_7^{\f}),  (\widetilde{A}_1^{\hf} + \widetilde{A}_7^{\hf})$
&
$8$
\\
\hline
$5$
&
$2 \times (\widetilde{A}_1^{\f} + \widetilde{A}_7^{\f}),  (2 \widetilde{A}_1^{\hf} + \widetilde{D}_6^{\f}),  (2 \widetilde{D}_4^{\f}),  (\widetilde{D}_8^{\f})$
&
$32$
\\
\hline
\end{tabular}
\end{table}

%----------------------

\subsubsection{Type IV}

$S$ has the following elliptic fibrations (this agrees with \cite[Proposition~8.9.19]{DK22}):
$$10 \times (2\widetilde{D}_4^{\f}),
40 \times (\widetilde{A}_3^{\hf}+\widetilde{D}_5^{\f}),
16 \times (2\widetilde{A}_4^{\f}),
16 \times (2\widetilde{A}_4^{\hf}),
5 \times (2\widetilde{A}_1^{\hf} + 2\widetilde{A}_3^{\f}).$$
The five fibrations of type $(2\widetilde{A}_1^{\hf}+2\widetilde{A}_3^{\f})$ are $(R_2+R_{10}, R_4+R_9, 1/2(R_{5}+R_{7}+R_{11}+R_{12}),1/2(R_{13}+R_{14}+R_{19}+R_{20}))$, $(R_{5}+R_{15}, R_{7}+R_{16}, 1/2(R_{9}+R_{10}+R_{17}+R_{18}),1/2(R_{1}+R_{3}+R_{13}+R_{14}))$, $(R_{17}+R_{19}, R_{18}+R_{20}, 1/2(R_{5}+R_{6}+R_{7}+R_{8}), 1/2(R_{1}+R_{2}+R_{3}+R_{4}))$, $(R_{6}+R_{13}, R_{8}+R_{14}, 1/2(R_{11}+R_{17}+R_{12}+R_{18}), 1/2(R_{2}+R_{15}+R_{4}+R_{16}))$, $(R_{1}+R_{11}, R_{3}+R_{12}, 1/2(R_{6}+R_{9}+R_{8}+R_{10}), 1/2(R_{15}+R_{19}+R_{16}+R_{20}))$. There are in total $64$ diagrams of type $\widetilde{A}_4$, $32$ of them are fibers and $32$ are half-fibers. In the notation of \cite{DK22}, they can be listed by choosing an element in $$\{R_1,R_3\} \times \{R_2,R_4\} \times \{R_{15},R_{16}\} \times \{R_{20},R_{19}\} \times \{R_{13},R_{14}\},$$
or an element in 
$$\{R_{11},R_{12}\} \times \{R_9,R_{10}\} \times \{R_{5},R_{7}\} \times \{R_{17},R_{18}\} \times \{R_{6},R_{8}\}.$$
Using the basis of $\Num(S)$ in Table~\ref{tbl:all-bases-Num-Enriques-finite-aut} it is possible to check which one of them is a fiber and which an half-fiber.

\begin{table}[h!]
\centering\renewcommand\cellalign{lc}
\setcellgapes{3pt}\makegapedcells
\caption{Saturated sequences on the Enriques surface \cite[(3.4)~Example~IV]{Kon86}.}
\label{Tab:sat-sequences-on-Kondo-IV}
\begin{tabular}{|c|c|c|}
\hline
Length
&
Fibrations in the sequence
&
Cardinality
\\
\hline
$10$
&
$5 \times (2 \widetilde{A}_1^{\hf} + 2 \widetilde{A}_3^{\f}),  5 \times (2 \widetilde{A}_4^{\f})$
&
$16$
\\
\hline
$9$
&
$5 \times (2 \widetilde{A}_1^{\hf} + 2 \widetilde{A}_3^{\f}),  4 \times (2 \widetilde{A}_4^{\f})$
&
$40$
\\
\hline
$9$
&
$4 \times (2 \widetilde{A}_1^{\hf} + 2 \widetilde{A}_3^{\f}),  4 \times (2 \widetilde{A}_4^{\f}),  (2 \widetilde{D}_4^{\f})$
&
$160$
\\
\hline
$8$
&
$2 \times (2 \widetilde{A}_1^{\hf} + 2 \widetilde{A}_3^{\f}),  (\widetilde{A}_3^{\hf} + \widetilde{D}_5^{\f}),  4 \times (2 \widetilde{A}_4^{\f}),  (2 \widetilde{D}_4^{\f})$
&
$80$
\\
\hline
$6$
&
$5 \times (2 \widetilde{A}_4^{\f}),  (2 \widetilde{A}_4^{\hf})$
&
$16$
\\
\hline
\end{tabular}
\end{table}

%----------------------

\subsubsection{Type V}

$S$ has the following elliptic fibrations (this agrees with \cite[Proposition~8.9.23]{DK22}):
$$4 \times (\widetilde{A}_1^{\hf} + \widetilde{A}_2^{\f} + \widetilde{A}_5^{\hf} ),
12 \times (\widetilde{A}_1^{\hf}+\widetilde{E}_7^{\f}),
6 \times (\widetilde{A}_1^{\f}+\widetilde{A}_1^{\hf}+\widetilde{D}_6^{\f}),
3 \times (\widetilde{A}_1^{\f}+\widetilde{A}_7^{\f}),
4 \times (\widetilde{A}_2^{\f}+\widetilde{E}_6^{\f}).$$
The four fibrations of type $(\widetilde{A}_1^{\hf} + \widetilde{A}_2^{\f} + \widetilde{A}_5^{\hf})$ are determined by a choice of a $\widetilde{A}_1^{\hf}$, given by a vertex of the tethrahedron $\{R_{11},R_{12},R_{13},R_{20}\}$ and the adjacent curve in $\{R_{1},R_{5},R_{6},R_{8}\}$. As an example we have
$(R_{1}+R_{12},1/2(R_{15}+R_{16}+R_{17}),R_{3}+R_{4}+R_5+R_7+R_8+R_9)$.

The six fibrations of type $(\widetilde{A}_1^{\f}+\widetilde{A}_1^{\hf}+\widetilde{D}_6^{\f})$ are determined by a choice of a $\widetilde{A}_1^{\f}$ being one of the red edges of the tethrahedron $\{R_{11},R_{12},R_{13},R_{20}\}$. As an example we have $(1/2(R_{11}+R_{12}),R_{10}+R_{15},1/2(R_{2}+2R_{3}+2R_4+2R_5+R_6+R_7+R_9))$.

The three fibrations of type $(\widetilde{A}_1^{\f}+\widetilde{A}_7^{\f})$ are determined by a choice of a $\widetilde{A}_1^{\f}$ being a diagonal of the octahedron $\{R_{14},R_{15},R_{16},R_{17},R_{18},R_{19}\}$. As an example we have $(1/2(R_{14}+R_{17}),1/2(R_{1}+R_{2}+R_3+R_4+R_5+R_7+R_8+R_{10}))$.

\begin{table}[h!]
\centering\renewcommand\cellalign{lc}
\setcellgapes{3pt}\makegapedcells
\caption{Saturated sequences on the Enriques surface \cite[(3.5)~Example~V]{Kon86}.}
\label{Tab:sat-sequences-on-Kondo-V}
\begin{tabular}{|c|c|c|}
\hline
Length
&
Fibrations in the sequence
&
Cardinality
\\
\hline
$7$
&
$3 \times (\widetilde{A}_1^{\f} + \widetilde{A}_1^{\hf} + \widetilde{D}_6^{\f}),  3 \times (\widetilde{A}_1^{\f} + \widetilde{A}_7^{\f}),  (\widetilde{A}_1^{\hf} + \widetilde{A}_2^{\f} + \widetilde{A}_5^{\hf})$
&
$4$
\\
\hline
$7$
&
$3 \times (\widetilde{A}_1^{\f} + \widetilde{A}_1^{\hf} + \widetilde{D}_6^{\f}),  3 \times (\widetilde{A}_1^{\f} + \widetilde{A}_7^{\f}),  (\widetilde{A}_2^{\f} + \widetilde{E}_6^{\f})$
&
$4$
\\
\hline
$7$
&
$2 \times (\widetilde{A}_1^{\f} + \widetilde{A}_1^{\hf} + \widetilde{D}_6^{\f}),  3 \times (\widetilde{A}_1^{\f} + \widetilde{A}_7^{\f}),  (\widetilde{A}_1^{\hf} + \widetilde{A}_2^{\f} + \widetilde{A}_5^{\hf}),  (\widetilde{A}_2^{\f} + \widetilde{E}_6^{\f})$
&
$12$
\\
\hline
$5$
&
$2 \times (\widetilde{A}_1^{\f} + \widetilde{A}_1^{\hf} + \widetilde{D}_6^{\f}),  (\widetilde{A}_1^{\f} + \widetilde{A}_7^{\f}),  (\widetilde{A}_1^{\hf} + \widetilde{E}_7^{\f}),  (\widetilde{A}_2^{\f} + \widetilde{E}_6^{\f})$
&
$12$
\\
\hline
\end{tabular}
\end{table}

%----------------------

\subsubsection{Type VI}

$S$ has the following elliptic fibrations (this agrees with \cite[Proposition~8.9.27]{DK22}):
$$
12 \times (\widetilde{A}_4^{\f} + \widetilde{A}_4^{\hf}),
10 \times (\widetilde{A}_1^{\hf}+\widetilde{A}_2^{\f} + \widetilde{A}_5^{\f}),
15 \times (\widetilde{A}_3^{\f}+\widetilde{D}_5^{\f}),
20 \times (\widetilde{A}_2^{\hf}+\widetilde{E}_6^{\f}).$$
The subgraph of Figure~\ref{fig_dual_graph_K6} induced by the rational curves $R_1,\ldots,R_{10}$ is a \textit{Petersen graph}, which implies that $\nd(S)=10$ (we direct the interested reader to Example~6.4.19 and \S\,8.9 of \cite{DK22}). Observe that the half-fibers listed above are numerically equivalent to fibers of type $\widetilde{A}_5$ divided by $2$ supported on the Petersen graph. For example, $R_1+R_{20}\equiv1/2(R_3+R_4+R_5+R_7+R_8+R_9)$. In fact, our computation shows that there is no other sequence of isotropic nef classes realizing $\nd(S)=10$. Equivalently, $S$ admits a unique ample Fano polarization.

\begin{table}[h!]
\centering\renewcommand\cellalign{lc}
\setcellgapes{3pt}\makegapedcells
\caption{Saturated sequences on the Enriques surface \cite[(3.6)~Example~VI]{Kon86}.}
\label{Tab:sat-sequences-on-Kondo-VI}
\begin{tabular}{|c|c|c|}
\hline
Length
&
Fibrations in the sequence
&
Cardinality
\\
\hline
$10$
&
$10 \times (\widetilde{A}_1^{\hf} + \widetilde{A}_2^{\f} + \widetilde{A}_5^{\f})$
&
$1$
\\
\hline
$9$
&
$8 \times (\widetilde{A}_1^{\hf} + \widetilde{A}_2^{\f} + \widetilde{A}_5^{\f}),  (\widetilde{A}_3^{\f} + \widetilde{D}_5^{\f})$
&
$15$
\\
\hline
$9$
&
$7 \times (\widetilde{A}_1^{\hf} + \widetilde{A}_2^{\f} + \widetilde{A}_5^{\f}),  2 \times (\widetilde{A}_3^{\f} + \widetilde{D}_5^{\f})$
&
$30$
\\
\hline
$9$
&
$6 \times (\widetilde{A}_1^{\hf} + \widetilde{A}_2^{\f} + \widetilde{A}_5^{\f}),  3 \times (\widetilde{A}_3^{\f} + \widetilde{D}_5^{\f})$
&
$10$
\\
\hline
$8$
&
$5 \times (\widetilde{A}_1^{\hf} + \widetilde{A}_2^{\f} + \widetilde{A}_5^{\f}),  2 \times (\widetilde{A}_3^{\f} + \widetilde{D}_5^{\f}),  (\widetilde{A}_4^{\f} + \widetilde{A}_4^{\hf})$
&
$60$
\\
\hline
$7$
&
$3 \times (\widetilde{A}_1^{\hf} + \widetilde{A}_2^{\f} + \widetilde{A}_5^{\f}),  (\widetilde{A}_2^{\hf} + \widetilde{E}_6^{\f}),  3 \times (\widetilde{A}_3^{\f} + \widetilde{D}_5^{\f})$
&
$20$
\\
\hline
\end{tabular}
\end{table}

%----------------------

\subsubsection{Type VII}

As also computed in \cite[Proposition~8.9.28]{DK22}, $S$ has the following elliptic fibrations:
$$20 \times (\widetilde{A}_8^{\f}),
15 \times (\widetilde{A}_1^{\hf}+\widetilde{A}_7^{\f}),
6 \times (2\widetilde{A}_4^{\f}),
10 \times (\widetilde{A}_1^{\f}+\widetilde{A}_2^{\hf}+\widetilde{A}_5^{\f}).
$$

\begin{table}[h!]
\centering\renewcommand\cellalign{lc}
\setcellgapes{3pt}\makegapedcells
\caption{Saturated sequences on the Enriques surface \cite[(3.7)~Example~VII]{Kon86}.}
\label{Tab:sat-sequences-on-Kondo-VII}
\begin{tabular}{|c|c|c|}
\hline
Length
&
Fibrations in the sequence
&
Cardinality
\\
\hline
$10$
&
$4 \times (\widetilde{A}_1^{\f} + \widetilde{A}_2^{\hf} + \widetilde{A}_5^{\f}), 6 \times (2 \widetilde{A}_4^{\f})$
&
$5$
\\
\hline
$9$
&
$4 \times (\widetilde{A}_1^{\f} + \widetilde{A}_2^{\hf} + \widetilde{A}_5^{\f}), (\widetilde{A}_1^{\hf} + \widetilde{A}_7^{\f}), 4 \times (2 \widetilde{A}_4^{\f})$
&
$15$
\\
\hline
$9$
&
$3 \times (\widetilde{A}_1^{\f} + \widetilde{A}_2^{\hf} + \widetilde{A}_5^{\f}), 6 \times (2 \widetilde{A}_4^{\f})$
&
$10$
\\
\hline
$7$
&
$3 \times (\widetilde{A}_1^{\f} + \widetilde{A}_2^{\hf} + \widetilde{A}_5^{\f}), 3 \times (2 \widetilde{A}_4^{\f}), (\widetilde{A}_8^{\f})$
&
$20$
\\
\hline
$7$
&
$2 \times (\widetilde{A}_1^{\f} + \widetilde{A}_2^{\hf} + \widetilde{A}_5^{\f}), (\widetilde{A}_1^{\hf} + \widetilde{A}_7^{\f}), 3 \times (2 \widetilde{A}_4^{\f}), (\widetilde{A}_8^{\f})$
&
$60$
\\
\hline
\end{tabular}
\end{table}

%-------------------------------------------

\subsection{Geometry of Fano models and Kuznetsov components}
\label{Geom-Fano-models-and-Kuznetsov-components}

With reference to \S\,\ref{ssec:SaturatedIsoSequences}, using the computational data produced in \S\,\ref{sec:data-on-elliptic-fibrations} we can exhibit explicit examples of non-isomorphic Fano models and Kuznetsov components. We specifically focus on the Enriques surfaces with finite automorphism group of type I and IV, but one can construct analogous examples in all types.

\begin{example}
\label{ex:nonEquivKuzComp}
Consider the Enriques surface with finite automorphism group of type IV. It follows from the data of Table~\ref{Tab:sat-sequences-on-Kondo-IV} and from the discussion at the end of \S\,\ref{ssec:SaturatedIsoSequences} that $S$ admits at least three non-isomorphic Fano models and three non-equivalent Kuznetsov components. These are obtained from sequences of length $10$, $8$, and $6$.
\end{example}

While we only give a simple example in this work, the problem of classifying Fano models and Kuznetsov components (and with them, canonical isotropic sequences) may provide interesting insights into the nature of $S$, and is left for future research.

Additionally, note that one can obtain non-isomorphic Fano models for an Enriques surface $S$ also by considering two different extensions to maximal canonical isotropic sequence of the same saturated non-degenerate sequence as the following example shows.

\begin{example}
\label{ex_Kondo1Extensions}
Consider an Enriques surface with finite automorphism group of type I. From Table~\ref{tbl:isotropic-sequences-realizing-cnd(S,R)} we know that $S$ admits a saturated isotropic sequence of length $4$ given by 
\begin{align*}
f_1&:=\frac{1}{2} [R_{10}+R_{11}], \ \ &
f_3:=\frac{1}{2}[2R_1+R_2+R_4+2R_5+2R_6+2R_7+2R_8+R_9+R_{12}],\\
f_2&:=[R_{9}+R_{10}], \ \ &
f_4:=\frac{1}{2}[2R_1+2R_2+2R_3+2R_4+2R_5+R_6+R_8+R_9+R_{12}].
\end{align*}
By Lemma~\ref{lem:ExtendingSequence}, $(f_1,f_2,f_3,f_4)$ can be extended to a canonical maximal isotropic sequence. A computer assisted inspection yields the following extensions of $(f_1, f_2, f_3, f_4)$ to canonical maximal isotropic sequences:
$$P\coloneqq (f_1, f_1+R_{12}, f_2, f_2+R_1, f_3, f_3+R_3,f_3+R_3+R_4, f_4, f_4+R_7,f_4+R_6+R_7),$$
$$Q\coloneqq (f_1, f_1+R_{9}, f_2, f_3, f_3+R_3,f_3+R_2+R_3, f_4, f_4+R_7,f_4+R_6+R_7,f_4+R_5+R_6+R_7),$$
where, for simplicity of notation, we identified the rational curves $R_i$ with their class in $\Num(S)$. Observe that the two sequences $P,Q$ define non-isomorphic Fano models $S_P$ and $S_Q$. In fact, $S_P$ has 4 singular points, two of type $A_1$ and two of type $A_2$, obtained by contracting the curves $R_1,R_3,R_4,R_6,R_7,R_{12}$. The Fano model $S_Q$ has three singular points of type $A_1$, $A_2$, and $A_3$, obtained contracting $R_2,R_3,R_5,R_6,R_7,R_9$.
\end{example}

%-------------------------------------------

\subsection{Configurations of smooth rational curves}
\label{sec:config-smooth-rat-curves}

Figures~\ref{fig_dual_graph_KonI}, \ref{fig_dual_graph_KonII}, \ref{fig_dual_graph_K3}, \ref{fig_dual_graph_K4}, \ref{fig_dual_graph_K5}, \ref{fig_dual_graph_K6}, and \ref{fig_dual_graph_K7} recollect the dual graphs of the smooth rational curves on the seven types of Enriques surfaces with finite automorphism group. In the figures, we adopt the following convention: a black (resp. colored) edge between two vertices indicates that the intersection of the corresponding curves equals $1$ (resp. $2$). For consistency of notation within this paper, the curves denoted by $E_i$ in \cite{DK22} will be denoted by $R_i$ instead. For the Enriques surface of type VII, the curves denoted by $K_i$, $i=1, \dots, 5$, in \cite{DK22}, will be denoted by $R_{i+15}$.

\newpage

\begin{figure}[H]
\begin{tikzpicture}
\draw  (0,0)-- (1,0);
\draw [color=ffqqqq] (1,0)-- (2,0);
\draw [color=ffqqqq] (2,0)-- (3,0);
\draw [color=ffqqqq] (3,0)-- (4,0);
\draw  (4,0)-- (5,0);
\draw  (0,0)-- (1,1);
\draw  (1,1)-- (2.5,2);
\draw  (2.5,2)-- (4,1);
\draw  (4,1)-- (5,0);
\draw  (5,0)-- (4,-1);
\draw  (4,-1)-- (2.5,-2);
\draw  (2.5,-2)-- (1,-1);
\draw  (1,-1)-- (0,0);
\fill [color=black] (0,0) circle (2.4pt);
\draw[color=black] (-0.3,0) node {$1$};
\fill [color=black] (1,0) circle (2.4pt);
\draw[color=black] (1.02,0.36) node {$9$};
\fill [color=black] (2,0) circle (2.4pt);
\draw[color=black] (2,0.36) node {$10$};
\fill [color=black] (3,0) circle (2.4pt);
\draw[color=black] (3,0.36) node {$11$};
\fill [color=black] (4,0) circle (2.4pt);
\draw[color=black] (4,0.36) node {$12$};
\fill [color=black] (5,0) circle (2.4pt);
\draw[color=black] (5.3,0) node {$5$};
\fill [color=black] (2.5,2) circle (2.4pt);
\draw[color=black] (2.5,2.29) node {$3$};
\fill [color=black] (2.5,-2) circle (2.4pt);
\draw[color=black] (2.48,-2.29) node {$7$};
\fill [color=black] (1,1) circle (2.4pt);
\draw[color=black] (0.81,1.31) node {$2$};
\fill [color=black] (4,1) circle (2.4pt);
\draw[color=black] (4.26,1.29) node {$4$};
\fill [color=black] (4,-1) circle (2.4pt);
\draw[color=black] (4.22,-1.16) node {$6$};
\fill [color=black] (1,-1) circle (2.4pt);
\draw[color=black] (0.79,-1.18) node {$8$};
\end{tikzpicture}
\caption{Configuration of $12$ smooth rational curves on the Enriques surface \cite[(3.2)~Example~I]{Kon86}.}
\label{fig_dual_graph_KonI}
\end{figure}

\begin{figure}[H]
\begin{tikzpicture}[scale=0.8]

	\draw (0,3) -- (1,2);
	\draw (0,3) -- (-1,2);
	\draw (0,1) -- (1,2);
	\draw (0,1) -- (-1,2);
	\draw (1,2) -- (2,1);
	\draw (2,1) -- (3,0);
	\draw (3,0) -- (2,-1);
	\draw (2,-1) -- (1,0);
	\draw (1,0) -- (2,1);
	\draw (2,-1) -- (-2,-1);
	\draw (-2,-1) -- (-3,0);
	\draw (-3,0) -- (-2,1);
	\draw (-2,1) -- (-1,0);
	\draw (-1,0) -- (-2,-1);
	\draw (-2,1) -- (-1,2);

	\fill (0,3) circle (2.74pt);
	\fill (0,1) circle (2.74pt);
	\fill (1,2) circle (2.74pt);
	\fill (-1,2) circle (2.74pt);
	\fill (2,1) circle (2.74pt);
	\fill (1,0) circle (2.74pt);
	\fill (3,0) circle (2.74pt);
	\fill (2,-1) circle (2.74pt);
	\fill (-2,-1) circle (2.74pt);
	\fill (-3,0) circle (2.74pt);
	\fill (-1,0) circle (2.74pt);
	\fill (-2,1) circle (2.74pt);

	\node at (-1-0.3,2+0.3) {$1$};
	\node at (0,3+0.4) {$2$};
	\node at (1+0.3,2+0.3) {$3$};
	\node at (0,1-0.4) {$4$};
	\node at (2+0.3,1+0.3) {$5$};
	\node at (3+0.4,0) {$6$};
	\node at (2,-1-0.4) {$7$};
	\node at (1-0.4,0) {$8$};
	\node at (-2,-1-0.4) {$9$};
	\node at (-3-0.4,0) {$10$};
	\node at (-2-0.3,1+0.3) {$11$};
	\node at (-1+0.4,0) {$12$};

\end{tikzpicture}
\caption{Configuration of $12$ smooth rational curves on the Enriques surface \cite[(3.2)~Example~II]{Kon86}.}
\label{fig_dual_graph_KonII}
\end{figure}

\begin{figure}[H]
\begin{tikzpicture}[scale=0.7]
\draw (-1,1)-- (-3,0);
\draw (-3,0)-- (-2,2);
\draw (-2,2)-- (0,3);
\draw (0,3)-- (-1,1);
\draw (0,3)-- (3,3);
\draw (0,3)-- (4,4);
\draw (3,3)-- (3,0);
\draw (4,4)-- (3,0);
\draw (-3,0)-- (-3,-3);
\draw (-3,0)-- (-4,-4);
\draw (-3,-3)-- (0,-3);
\draw (0,-3)-- (-4,-4);
\draw (2,-2)-- (0,-3);
\draw (0,-3)-- (1,-1);
\draw (1,-1)-- (3,0);
\draw (3,0)-- (2,-2);
\draw [color=ffqqqq] (-1,1)-- (2,1);
\draw [color=ffqqqq] (-1,1)-- (-1,-1);
\draw [color=ffqqqq] (-2,2)-- (1,2);
\draw [color=ffqqqq] (-2,2)-- (-2,-2);
\draw [color=ffqqqq] (-2,-2)-- (-1,-1);
\draw [color=ffqqqq] (2,1)-- (1,2);
\draw [color=ffqqqq] (2,1)-- (2,-2);
\draw [color=ffqqqq] (2,-2)-- (-2,-2);
\draw [color=ffqqqq] (-1,-1)-- (1,-1);
\draw [color=ffqqqq] (1,-1)-- (1,2);
\draw [color=ffqqqq] (-5,5)-- (-4,-4);
\draw [color=ffqqqq] (-5,5)-- (4,4);
\draw [color=ffqqqq] (-4,4)-- (-5,5);
\draw [color=ffqqqq] (-4,4)-- (-3,-3);
\draw [color=ffqqqq] (-4,4)-- (3,3);
\draw [color=ffqqqq] (6,-4)-- (5,-5);
\draw [color=ffqqqq] (5,-5)-- (-4,-4);
\draw [color=ffqqqq] (5,-5)-- (3,3);
\draw [color=ffqqqq] (6,-4)-- (-3,-3);
\draw [color=ffqqqq] (6,-4)-- (4,4);
\draw [color=cccccc] (-4,4)-- (1,2);
\draw [color=cccccc] (-4,4)-- (2,1);
\draw [color=cccccc] (-4,4)-- (-1,-1);
\draw [color=cccccc] (-4,4)-- (-2,-2);
\draw [color=cccccc] (-5,5)-- (1,2);
\draw [color=cccccc] (-5,5)-- (2,1);
\draw [color=cccccc] (-5,5)-- (-1,-1);
\draw [color=cccccc] (-5,5)-- (-2,-2);
\draw [color=cccccc] (6,-4)-- (1,2);
\draw [color=cccccc] (2,1)-- (6,-4);
\draw [color=cccccc] (6,-4)-- (-2,-2);
\draw [color=cccccc] (-2,-2)-- (5,-5);
\draw [color=cccccc] (5,-5)-- (-1,-1);
\draw [color=cccccc] (6,-4)-- (-1,-1);
\draw [color=cccccc] (5,-5)-- (2,1);
\draw [color=cccccc] (5,-5)-- (1,2);

\fill [color=black] (0,3) circle (2.74pt);
\draw[color=black] (0,3.36) node {$1$};
\fill [color=black] (-1,1) circle (2.74pt);
\draw[color=black] (-1.29,1.48) node {$10$};
\fill [color=black] (-2,2) circle (2.74pt);
\draw[color=black] (-2.22,2.37) node {$8$};
\fill [color=black] (-3,0) circle (2.74pt);
\draw[color=black] (-3.2,0.29) node {$7$};
\fill [color=black] (3,3) circle (2.74pt);
\draw[color=black] (3.23,3.27) node {$9$};
\fill [color=black] (4,4) circle (2.74pt);
\draw[color=black] (4.2,4.25) node {$2$};
\fill [color=black] (3,0) circle (2.74pt);
\draw[color=black] (3.33,-0.09) node {$3$};
\fill [color=black] (-3,-3) circle (2.74pt);
\draw[color=black] (-3.17,-3.3) node {$6$};
\fill [color=black] (-4,-4) circle (2.74pt);
\draw[color=black] (-4.16,-4.4) node {$12$};
\fill [color=black] (0,-3) circle (2.74pt);
\draw[color=black] (0.3,-3.15) node {$5$};
\fill [color=black] (2,-2) circle (2.74pt);
\draw[color=black] (2.43,-2.08) node {$11$};
\fill [color=black] (1,-1) circle (2.74pt);
\draw[color=black] (1.33,-1.07) node {$4$};
\fill [color=black] (1,2) circle (2.74pt);
\draw[color=black] (1.28,2.26) node {$18$};
\fill [color=black] (2,1) circle (2.74pt);
\draw[color=black] (2.39,1.4) node {$14$};
\fill [color=black] (-2,-2) circle (2.74pt);
\draw[color=black] (-2.16,-2.4) node {$13$};
\fill [color=black] (-1,-1) circle (2.74pt);
\draw[color=black] (-0.54,-0.53) node {$17$};
\fill [color=black] (5,-5) circle (2.74pt);
\draw[color=black] (5.5,-5) node {$16$};
\fill [color=black] (6,-4) circle (2.74pt);
\draw[color=black] (6.45,-4) node {$20$};
\fill [color=black] (-4,4) circle (2.74pt);
\draw[color=black] (-3.47,3.6) node {$19$};
\fill [color=black] (-5,5) circle (2.74pt);
\draw[color=black] (-5.4,5) node {$15$};
\end{tikzpicture}
\caption{Dual graph of the rational curves $R_1, \ldots, R_{20}$ in the Enriques surface of type III in \cite{DK22}. Every vertex in $\{R_{15},R_{16},R_{19},R_{20}\}$ is connected to every vertex in $\{R_{13},R_{14},R_{17},R_{18}\}$ via the blue edges.} 
\label{fig_dual_graph_K3}
\end{figure}

\begin{figure}
\begin{tikzpicture}
\draw [color=ffqqqq] (1.35,4.84)-- (1.15,2.67);
\draw [color=ffqqqq] (0.65,4.84)-- (0.85,2.67);
\draw [color=ffqqqq] (4.56,2.5)-- (2.57,1.64);
\draw [color=ffqqqq] (4.78,1.84)-- (2.66,1.36);
\draw [color=ffqqqq] (3.55,-1.94)-- (2.12,-0.31);
\draw [color=ffqqqq] (1.87,-0.48)-- (2.99,-2.35);
\draw [color=ffqqqq] (-0.99,-2.35)-- (0.13,-0.48);
\draw [color=ffqqqq] (-0.12,-0.31)-- (-1.55,-1.94);
\draw [color=ffqqqq] (-2.78,1.84)-- (-0.66,1.36);
\draw [color=ffqqqq] (-0.57,1.64)-- (-2.56,2.5);
\draw (1.35,4.84)-- (4.56,2.5);
\draw (0.65,4.84)-- (4.78,1.84);
\draw (1.35,4.84)-- (4.78,1.84);
\draw (4.56,2.5)-- (0.65,4.84);
\draw (0.65,4.84)-- (-2.56,2.5);
\draw (-2.56,2.5)-- (1.35,4.84);
\draw (1.35,4.84)-- (-2.78,1.84);
\draw (-2.78,1.84)-- (0.65,4.84);
\draw (-2.78,1.84)-- (-1.55,-1.94);
\draw (-1.55,-1.94)-- (-2.56,2.5);
\draw (-2.56,2.5)-- (-0.99,-2.35);
\draw (-0.99,-2.35)-- (-2.78,1.84);
\draw (-0.99,-2.35)-- (2.99,-2.35);
\draw (2.99,-2.35)-- (-1.55,-1.94);
\draw (-1.55,-1.94)-- (3.55,-1.94);
\draw (3.55,-1.94)-- (-0.99,-2.35);
\draw (3.55,-1.94)-- (4.78,1.84);
\draw (4.78,1.84)-- (2.99,-2.35);
\draw (2.99,-2.35)-- (4.56,2.5);
\draw (4.56,2.5)-- (3.55,-1.94);
\draw (1.15,2.67)-- (2.12,-0.31);
\draw (2.12,-0.31)-- (0.85,2.67);
\draw (0.85,2.67)-- (1.87,-0.48);
\draw (1.15,2.67)-- (1.87,-0.48);
\draw (0.85,2.67)-- (-0.12,-0.31);
\draw (-0.12,-0.31)-- (1.15,2.67);
\draw (1.15,2.67)-- (0.13,-0.48);
\draw (0.13,-0.48)-- (0.85,2.67);
\draw (0.13,-0.48)-- (2.66,1.36);
\draw (2.66,1.36)-- (-0.12,-0.31);
\draw (-0.12,-0.31)-- (2.57,1.64);
\draw (2.57,1.64)-- (0.13,-0.48);
\draw (2.57,1.64)-- (-0.57,1.64);
\draw (-0.57,1.64)-- (2.66,1.36);
\draw (2.66,1.36)-- (-0.66,1.36);
\draw (-0.66,1.36)-- (2.57,1.64);
\draw (-0.66,1.36)-- (1.87,-0.48);
\draw (1.87,-0.48)-- (-0.57,1.64);
\draw (-0.57,1.64)-- (2.12,-0.31);
\draw (2.12,-0.31)-- (-0.66,1.36);
\fill [color=black] (1.35,4.84) circle (1.9pt);
\draw[color=black] (1.56,5.2) node {$2$};
\fill [color=black] (1.15,2.67) circle (1.9pt);
\draw[color=black] (1.5,2.8) node {$10$};
\fill [color=black] (0.65,4.84) circle (1.9pt);
\draw[color=black] (0.43,5.2) node {$4$};
\fill [color=black] (0.85,2.67) circle (1.9pt);
\draw[color=black] (0.51,2.8) node {$9$};
\fill [color=black] (4.56,2.5) circle (1.9pt);
\draw[color=black] (4.9,2.7) node {$15$};
\fill [color=black] (2.57,1.64) circle (1.9pt);
\draw[color=black] (2.5,2) node {$5$};
\fill [color=black] (4.78,1.84) circle (1.9pt);
\draw[color=black] (5.2,1.87) node {$16$};
\fill [color=black] (2.66,1.36) circle (1.9pt);
\draw[color=black] (2.9,1.15) node {$7$};
\fill [color=black] (3.55,-1.94) circle (1.9pt);
\draw[color=black] (3.9,-2.17) node {$19$};
\fill [color=black] (2.12,-0.31) circle (1.9pt);
\draw[color=black] (2.59,-0.17) node {$17$};
\fill [color=black] (1.87,-0.48) circle (1.9pt);
\draw[color=black] (1.74,-0.83) node {$18$};
\fill [color=black] (2.99,-2.35) circle (1.9pt);
\draw[color=black] (3.3,-2.6) node {$20$};
\fill [color=black] (-0.99,-2.35) circle (1.9pt);
\draw[color=black] (-1.25,-2.6) node {$13$};
\fill [color=black] (0.13,-0.48) circle (1.9pt);
\draw[color=black] (0.33,-0.78) node {$6$};
\fill [color=black] (-0.12,-0.31) circle (1.9pt);
\draw[color=black] (-0.45,-0.11) node {$8$};
\fill [color=black] (-1.55,-1.94) circle (1.9pt);
\draw[color=black] (-1.85,-2.17) node {$14$};
\fill [color=black] (-2.78,1.84) circle (1.9pt);
\draw[color=black] (-3.1,1.84) node {$1$};
\fill [color=black] (-0.66,1.36) circle (1.9pt);
\draw[color=black] (-1,1.15) node {$11$};
\fill [color=black] (-0.57,1.64) circle (1.9pt);
\draw[color=black] (-0.5,2) node {$12$};
\fill [color=black] (-2.56,2.5) circle (1.9pt);
\draw[color=black] (-2.8,2.7) node {$3$};
\end{tikzpicture}
\caption{Dual graph of the rational curves $R_1, \ldots, R_{20}$ in the Enriques surface of type IV in \cite{DK22}.
} 
\label{fig_dual_graph_K4}
\end{figure}

\begin{figure}
\begin{tikzpicture}[scale=0.6]
\draw (14,4)-- (11,1);
\draw (11,1)-- (6,-4);
\draw (6,-4)-- (6,8);
\draw (6,8)-- (6,12);
\draw (6,12)-- (1,7);
\draw (1,7)-- (-2,4);
\draw (-2,4)-- (6,-6);
\draw (6,-6)-- (14,4);
\draw (14,4)-- (11,7);
\draw (11,7)-- (6,12);
\draw (6,-4)-- (1,1);
\draw (1,1)-- (-2,4);
\draw (5,2)-- (3.5,6.5);
\draw (3.5,6.5)-- (7,6);
\draw (7,6)-- (8.5,1.5);
\draw (8.5,1.5)-- (5,2);
\draw (3.5,6.5)-- (3.5,1.5);
\draw (3.5,1.5)-- (7,6);
\draw (5,2)-- (3.5,1.5);
\draw (3.5,1.5)-- (8.5,1.5);
\draw (5,2)-- (8.5,6.5);
\draw (8.5,6.5)-- (7,6);
\draw (8.5,6.5)-- (8.5,1.5);
\draw (8.5,6.5)-- (3.5,6.5);
\draw [color=cccccc] (6,-6)-- (5,2);
\draw [color=cccccc] (3.5,6.5)-- (11,1);
\draw [color=cccccc] (7,6)-- (6,8);
\draw [color=cccccc] (8.5,1.5)-- (1,7);
\draw [color=cccccc] (8.5,6.5)-- (1,1);
\draw [color=cccccc] (3.5,1.5)-- (11,7);
\draw [color=cccccc] (8.5,6.5)-- (3.5,1.5);
\draw [color=cccccc] (3.5,6.5)-- (8.5,1.5);
\draw [color=cccccc] (5,2)-- (7,6);
\draw [color=ffqqqq] (12,5)-- (7,-2);
\draw [color=ffqqqq] (7,-2)-- (5,10);
\draw [color=ffqqqq] (5,10)-- (0,3);
\draw [color=ffqqqq] (0,3)-- (12,5);
\draw [color=ffqqqq] (12,5)-- (5,10);
\draw [color=ffqqqq] (0,3)-- (7,-2);
\draw [color=cccccc] (14,4)-- (12,5);
\draw [color=cccccc] (12,5)-- (8.5,1.5);
\draw [color=cccccc] (12,5)-- (7,6);
\draw [color=cccccc] (0,3)-- (-2,4);
\draw [color=cccccc] (12,5)-- (8.5,6.5);
\draw [color=cccccc] (0,3)-- (7,6);
\draw [color=cccccc] (0,3)-- (3.5,6.5);
\draw [color=cccccc] (0,3)-- (3.5,1.5);
\draw [color=cccccc] (5,10)-- (6,12);
\draw [color=cccccc] (5,10)-- (8.5,6.5);
\draw [color=cccccc] (5,10)-- (3.5,6.5);
\draw [color=cccccc] (5,10)-- (5,2);
\draw [color=cccccc] (6,-4)-- (7,-2);
\draw [color=cccccc] (7,-2)-- (8.5,1.5);
\draw [color=cccccc] (7,-2)-- (5,2);
\draw [color=cccccc] (7,-2)-- (3.5,1.5);
\fill [color=black] (14,4) circle (3.2pt);
\draw[color=black] (14.36,4) node {$1$};
\fill [color=black] (11,1) circle (3.2pt);
\draw[color=black] (11,1.6) node {$2$};
\fill [color=black] (6,-4) circle (3.2pt);
\draw[color=black] (6.5,-4) node {$3$};
\fill [color=black] (6,8) circle (3.2pt);
\draw[color=black] (6.37,8.2) node {$9$};
\fill [color=black] (6,12) circle (3.2pt);
\draw[color=black] (6.4,12.3) node {$8$};
\fill [color=black] (1,7) circle (3.2pt);
\draw[color=black] (0.8,7.4) node {$7$};
\fill [color=black] (-2,4) circle (3.2pt);
\draw[color=black] (-2.4,4) node {$5$};
\fill [color=black] (6,-6) circle (3.2pt);
\draw[color=black] (6.5,-6) node {$6$};
\fill [color=black] (11,7) circle (3.2pt);
\draw[color=black] (11.5,7.4) node {{$10$}};
\fill [color=black] (1,1) circle (3.2pt);
\draw[color=black] (1.5,1) node {$4$};
\fill [color=black] (5,2) circle (3.2pt);
\draw[color=black] (5.5,2.2) node {{$17$}};
\fill [color=black] (3.5,6.5) circle (3.2pt);
\draw[color=black] (4.15,6.9) node {{$16$}};
\fill [color=black] (7,6) circle (3.2pt);
\draw[color=black] (7.1,6.32) node {{$14$}};
\fill [color=black] (8.5,1.5) circle (3.2pt);
\draw[color=black] (8.03,1.2) node {{$19$}};
\fill [color=black] (3.5,1.5) circle (3.2pt);
\draw[color=black] (3.2,1.1) node {{$15$}};
\fill [color=black] (8.5,6.5) circle (3.2pt);
\draw[color=black] (8.9,6.8) node {{$18$}};
\fill [color=black] (12,5) circle (3.2pt);
\draw[color=black] (12.1,4.55) node {{$12$}};
\fill [color=black] (7,-2) circle (3.2pt);
\draw[color=black] (7.5,-1.9) node {{$13$}};
\fill [color=black] (5,10) circle (3.2pt);
\draw[color=black] (5.5,10.2) node {{$11$}};
\fill [color=black] (0,3) circle (3.2pt);
\draw[color=black] (-0.2,3.6) node {{$20$}};
\end{tikzpicture}
\caption{Dual graph of the rational curves $R_1, \ldots, R_{20}$ in the Enriques surface of type V in \cite{DK22}.}
\label{fig_dual_graph_K5}
\end{figure}

\begin{figure}
\definecolor{uququq}{rgb}{0.1,0.1,0.1}
\begin{tikzpicture}[scale=0.2]
\draw  (2.5,7.55)-- (4.91,0.12);
\draw  (4.91,0.12)-- (-1.41,4.71);
\draw  (-1.41,4.71)-- (6.41,4.71);
\draw  (6.41,4.71)-- (0.09,0.12);
\draw  (0.09,0.12)-- (2.5,7.55);
\draw  (2.5,9.5)-- (8.26,5.31);
\draw  (8.26,5.31)-- (6.06,-1.46);
\draw  (6.06,-1.46)-- (-1.06,-1.46);
\draw  (-1.06,-1.46)-- (-3.26,5.31);
\draw  (-3.26,5.31)-- (2.5,9.5);
\draw  (2.5,9.5)-- (2.5,7.55);
\draw  (8.26,5.31)-- (6.41,4.71);
\draw  (6.06,-1.46)-- (4.91,0.12);
\draw  (-1.06,-1.46)-- (0.09,0.12);
\draw  (-3.26,5.31)-- (-1.41,4.71);
\draw  (2.5,24.93)-- (15.13,-13.94);
\draw  (15.13,-13.94)-- (-17.93,10.08);
\draw  (-17.93,10.08)-- (22.93,10.08);
\draw  (22.93,10.08)-- (-10.13,-13.94);
\draw  (-10.13,-13.94)-- (2.5,24.93);
\draw [color=ffqqqq] (2.5,24.93)-- (2.5,9.5);
\draw [color=ffqqqq] (22.93,10.08)-- (8.26,5.31);
\draw [color=ffqqqq] (15.13,-13.94)-- (6.06,-1.46);
\draw [color=ffqqqq] (-10.13,-13.94)-- (-1.06,-1.46);
\draw [color=ffqqqq] (-17.93,10.08)-- (-3.26,5.31);
\draw  (-6.94,16.44)-- (2.5,24.93);
\draw  (2.5,24.93)-- (11.94,16.44);
\draw  (11.94,16.44)-- (22.93,10.08);
\draw  (22.93,10.08)-- (17.78,-1.52);
\draw  (17.78,-1.52)-- (15.13,-13.94);
\draw  (15.13,-13.94)-- (2.5,-12.62);
\draw  (2.5,-12.62)-- (-10.13,-13.94);
\draw  (-10.13,-13.94)-- (-12.78,-1.52);
\draw  (-12.78,-1.52)-- (-17.93,10.08);
\draw  (-17.93,10.08)-- (-6.94,16.44);
\draw  (-6.94,16.44)-- (11.94,16.44);
\draw  (11.94,16.44)-- (17.78,-1.52);
\draw  (17.78,-1.52)-- (2.5,-12.62);
\draw  (2.5,-12.62)-- (-12.78,-1.52);
\draw  (-12.78,-1.52)-- (-6.94,16.44);
\draw [color=uququq] (-6.94,16.44)-- (-10.13,-13.94);
\draw [color=uququq] (-6.94,16.44)-- (22.93,10.08);
\draw [color=uququq] (11.94,16.44)-- (-17.93,10.08);
\draw [color=uququq] (11.94,16.44)-- (15.13,-13.94);
\draw [color=uququq] (17.78,-1.52)-- (2.5,24.93);
\draw [color=uququq] (17.78,-1.52)-- (-10.13,-13.94);
\draw [color=uququq] (2.5,-12.62)-- (22.93,10.08);
\draw [color=uququq] (2.5,-12.62)-- (-17.93,10.08);
\draw [color=uququq] (-12.78,-1.52)-- (2.5,24.93);
\draw [color=uququq] (-12.78,-1.52)-- (15.13,-13.94);
\draw [color=ffqqqq] (4.91,0.12)-- (-6.94,16.44);
\draw [color=ffqqqq] (11.94,16.44)-- (0.09,0.12);
\draw [color=ffqqqq] (17.78,-1.52)-- (-1.41,4.71);
\draw [color=ffqqqq] (2.5,-12.62)-- (2.5,7.55);
\draw [color=ffqqqq] (-12.78,-1.52)-- (6.41,4.71);
\fill [color=uququq] (2.5,7.55) circle (9.6pt);
\draw[color=uququq] (3.57,7.61) node {$1$};
\fill [color=uququq] (6.41,4.71) circle (9.6pt);
\draw[color=uququq] (7,3.9) node {$7$};
\fill [color=uququq] (4.91,0.12) circle (9.6pt);
\draw[color=uququq] (4,-0.7) node {$10$};
\fill [color=uququq] (0.09,0.12) circle (9.6pt);
\draw[color=uququq] (-0.8,0.6) node {$2$};
\fill [color=uququq] (-1.41,4.71) circle (9.6pt);
\draw[color=uququq] (-0.9,5.73) node {$8$};
\fill [color=uququq] (2.5,9.5) circle (9.6pt);
\draw[color=uququq] (3.7,9.4) node {$6$};
\fill [color=uququq] (8.26,5.31) circle (9.6pt);
\draw[color=uququq] (8.5,4) node {$5$};
\fill [color=uququq] (6.06,-1.46) circle (9.6pt);
\draw[color=uququq] (6.3,-2.9) node {$4$};
\fill [color=uququq] (-1.06,-1.46) circle (9.6pt);
\draw[color=uququq] (-2.4,-2) node {$3$};
\fill [color=uququq] (-3.26,5.31) circle (9.6pt);
\draw[color=uququq] (-4.4,6.6) node {$9$};
\fill [color=uququq] (2.5,24.93) circle (9.6pt);
\draw[color=uququq] (2.5,26) node {$19$};
\fill [color=uququq] (22.93,10.08) circle (9.6pt);
\draw[color=uququq] (24.03,11.1) node {$13$};
\fill [color=uququq] (15.13,-13.94) circle (9.6pt);
\draw[color=uququq] (16,-14.9) node {$18$};
\fill [color=uququq] (-10.13,-13.94) circle (9.6pt);
\draw[color=uququq] (-11.1,-14.9) node {$17$};
\fill [color=uququq] (-17.93,10.08) circle (9.6pt);
\draw[color=uququq] (-18.9,11.1) node {$15$};
\fill [color=uququq] (-6.94,16.44) circle (9.6pt);
\draw[color=uququq] (-8.17,17.77) node {$16$};
\fill [color=uququq] (11.94,16.44) circle (9.6pt);
\draw[color=uququq] (13.33,17.37) node {$12$};
\fill [color=uququq] (17.78,-1.52) circle (9.6pt);
\draw[color=uququq] (19.37,-1.8) node {$11$};
\fill [color=uququq] (2.5,-12.62) circle (9.6pt);
\draw[color=uququq] (2.58,-13.6) node {$20$};
\fill [color=uququq] (-12.78,-1.52) circle (9.6pt);
\draw[color=uququq] (-14.1,-1.8) node {$14$};
\end{tikzpicture}
\caption{Dual graph of the rational curves $R_1, \ldots, R_{20}$ in the Enriques surface of type VI in \cite{DK22}.}
\label{fig_dual_graph_K6}
\end{figure}

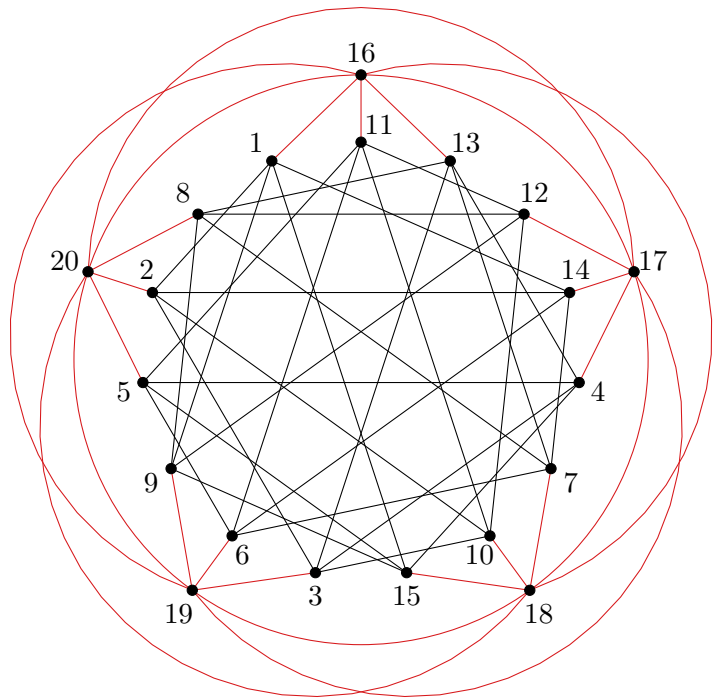
\begin{figure}
\begin{tikzpicture}[scale=0.6]
\clip(-8.128935855736028,-2.9324554936934453) rectangle (8.477993427868208,12.708954645515197);
\draw (-2.3,9.5) node {$1$};
\draw (-4.7,6.7) node {$2$};
\draw (-1,-0.5) node {$3$};
\draw (5.2,4) node {$4$};
\draw (-5.2,4) node {$5$};
\draw (-2.6,0.4) node {$6$};
\draw (4.6,2) node {$7$};
\draw (-3.9,8.4) node {$8$};
\draw (-4.6,2) node {$9$};
\draw (2.6,0.4) node {$10$};
\draw (0.4,9.9) node {$11$};
\draw (3.8,8.4) node {$12$};
\draw (2.3,9.5) node {$13$};
\draw (4.7,6.7) node {$14$};
\draw (1,-0.5) node {$15$};
\draw (0,11.5) node {$16$};
\draw (6.4,6.9) node {$17$};
\draw (3.9,-0.9) node {$18$};
\draw (-4,-0.9) node {$19$};
\draw (-6.5,6.9) node {$20$};
\draw [line width=0.4pt,color=ffqqqq] (0.,11.)-- (0.,9.514364454222578);
\draw [line width=0.4pt,color=ffqqqq] (0.,11.)-- (-1.9562952014676087,9.09854107258706);
\draw [line width=0.4pt,color=ffqqqq] (0.,11.)-- (1.9562952014676123,9.098541072587059);
\draw [line width=0.4pt,color=ffqqqq] (-5.987252556868822,6.650006391525967)-- (-3.5743291902175023,7.922970568002118);
\draw [line width=0.4pt,color=ffqqqq] (-5.987252556868822,6.650006391525967)-- (-4.574329190217502,6.190919760433243);
\draw [line width=0.4pt,color=ffqqqq] (-5.987252556868822,6.650006391525967)-- (-4.78338611675281,4.201875969696695);
\draw [line width=0.4pt,color=ffqqqq] (-3.70032557937465,-0.3884311178298277)-- (-4.165352128002917,2.29976293710639);
\draw [line width=0.4pt,color=ffqqqq] (-3.70032557937465,-0.3884311178298277)-- (-2.8270909152852033,0.813473286151603);
\draw [line width=0.4pt,color=ffqqqq] (-3.70032557937465,-0.3884311178298277)-- (-1.,0.);
\draw [line width=0.4pt,color=ffqqqq] (1.,0.)-- (3.7003255793746446,-0.38843111782983275);
\draw [line width=0.4pt,color=ffqqqq] (2.8270909152851993,0.8134732861515994)-- (3.7003255793746446,-0.38843111782983275);
\draw [line width=0.4pt,color=ffqqqq] (4.165352128002915,2.2997629371063866)-- (3.7003255793746446,-0.38843111782983275);
\draw [line width=0.4pt,color=ffqqqq] (5.987252556868827,6.6500063915259595)-- (4.78338611675281,4.201875969696692);
\draw [line width=0.4pt,color=ffqqqq] (5.987252556868827,6.6500063915259595)-- (4.574329190217504,6.190919760433237);
\draw [line width=0.4pt,color=ffqqqq] (5.987252556868827,6.6500063915259595)-- (3.574329190217505,7.922970568002114);
\draw [line width=0.4pt,color=ffqqqq] (0.,4.704630109478453) circle (6.295369890521547cm);
\draw [shift={(-1.7074982335406321,5.259429916838701)},line width=0.4pt,color=ffqqqq]  plot[domain=1.281686116022375:4.373180660439252,variable=\t]({1.*5.989131414256237*cos(\t r)+0.*5.989131414256237*sin(\t r)},{0.*5.989131414256237*cos(\t r)+1.*5.989131414256237*sin(\t r)});
\draw [shift={(-1.0552919440585171,3.252145356857434)},line width=0.4pt,color=ffqqqq]  plot[domain=2.5383231774582913:5.6298177218751695,variable=\t]({1.*5.989131414256237*cos(\t r)+0.*5.989131414256237*sin(\t r)},{0.*5.989131414256237*cos(\t r)+1.*5.989131414256237*sin(\t r)});
\draw [shift={(1.0552919440585176,3.252145356857432)},line width=0.4pt,color=ffqqqq]  plot[domain=-2.488225068285378:0.6032694761315007,variable=\t]({1.*5.989131414256239*cos(\t r)+0.*5.989131414256239*sin(\t r)},{0.*5.989131414256239*cos(\t r)+1.*5.989131414256239*sin(\t r)});
\draw [shift={(1.7074982335406343,5.259429916838698)},line width=0.4pt,color=ffqqqq]  plot[domain=-1.23158800684946:1.8599065375674189,variable=\t]({1.*5.989131414256236*cos(\t r)+0.*5.989131414256236*sin(\t r)},{0.*5.989131414256236*cos(\t r)+1.*5.989131414256236*sin(\t r)});
\draw [shift={(0.,6.5)},line width=0.4pt,color=ffqqqq]  plot[domain=0.02504905458645669:3.1165435990033354,variable=\t]({1.*5.98913141425624*cos(\t r)+0.*5.98913141425624*sin(\t r)},{0.*5.98913141425624*cos(\t r)+1.*5.98913141425624*sin(\t r)});
\draw [line width=0.4pt] (-1.9562952014676087,9.09854107258706)-- (-4.574329190217502,6.190919760433243);
\draw [line width=0.4pt] (-1.9562952014676087,9.09854107258706)-- (-4.165352128002917,2.29976293710639);
\draw [line width=0.4pt] (-1.9562952014676087,9.09854107258706)-- (1.,0.);
\draw [line width=0.4pt] (-1.9562952014676087,9.09854107258706)-- (4.574329190217504,6.190919760433237);
\draw [line width=0.4pt] (-4.574329190217502,6.190919760433243)-- (4.574329190217504,6.190919760433237);
\draw [line width=0.4pt] (-4.574329190217502,6.190919760433243)-- (2.8270909152851993,0.8134732861515994);
\draw [line width=0.4pt] (-4.574329190217502,6.190919760433243)-- (-1.,0.);
\draw [line width=0.4pt] (-1.,0.)-- (1.9562952014676123,9.098541072587059);
\draw [line width=0.4pt] (-1.,0.)-- (4.78338611675281,4.201875969696692);
\draw [line width=0.4pt] (-1.,0.)-- (2.8270909152851993,0.8134732861515994);
\draw [line width=0.4pt] (4.78338611675281,4.201875969696692)-- (1.,0.);
\draw [line width=0.4pt] (4.78338611675281,4.201875969696692)-- (-4.78338611675281,4.201875969696695);
\draw [line width=0.4pt] (4.78338611675281,4.201875969696692)-- (1.9562952014676123,9.098541072587059);
\draw [line width=0.4pt] (-4.78338611675281,4.201875969696695)-- (0.,9.514364454222578);
\draw [line width=0.4pt] (-4.78338611675281,4.201875969696695)-- (1.,0.);
\draw [line width=0.4pt] (-4.78338611675281,4.201875969696695)-- (-2.8270909152852033,0.813473286151603);
\draw [line width=0.4pt] (-2.8270909152852033,0.813473286151603)-- (4.165352128002915,2.2997629371063866);
\draw [line width=0.4pt] (-2.8270909152852033,0.813473286151603)-- (4.574329190217504,6.190919760433237);
\draw [line width=0.4pt] (-2.8270909152852033,0.813473286151603)-- (0.,9.514364454222578);
\draw [line width=0.4pt] (4.165352128002915,2.2997629371063866)-- (-3.5743291902175023,7.922970568002118);
\draw [line width=0.4pt] (4.165352128002915,2.2997629371063866)-- (1.9562952014676123,9.098541072587059);
\draw [line width=0.4pt] (4.165352128002915,2.2997629371063866)-- (4.574329190217504,6.190919760433237);
\draw [line width=0.4pt] (-3.5743291902175023,7.922970568002118)-- (-4.165352128002917,2.29976293710639);
\draw [line width=0.4pt] (-3.5743291902175023,7.922970568002118)-- (3.574329190217505,7.922970568002114);
\draw [line width=0.4pt] (-3.5743291902175023,7.922970568002118)-- (1.9562952014676123,9.098541072587059);
\draw [line width=0.4pt] (-4.165352128002917,2.29976293710639)-- (3.574329190217505,7.922970568002114);
\draw [line width=0.4pt] (-4.165352128002917,2.29976293710639)-- (1.,0.);
\draw [line width=0.4pt] (2.8270909152851993,0.8134732861515994)-- (0.,9.514364454222578);
\draw [line width=0.4pt] (2.8270909152851993,0.8134732861515994)-- (3.574329190217505,7.922970568002114);
\draw [line width=0.4pt] (0.,9.514364454222578)-- (3.574329190217505,7.922970568002114);
\draw [fill=black] (-1.,0.) circle (3.2pt);
\draw [fill=black] (1.,0.) circle (3.2pt);
\draw [fill=black] (2.8270909152851993,0.8134732861515994) circle (3.2pt);
\draw [fill=black] (4.165352128002915,2.2997629371063866) circle (3.2pt);
\draw [fill=black] (4.78338611675281,4.201875969696692) circle (3.2pt);
\draw [fill=black] (4.574329190217504,6.190919760433237) circle (3.2pt);
\draw [fill=black] (3.574329190217505,7.922970568002114) circle (3.2pt);
\draw [fill=black] (1.9562952014676123,9.098541072587059) circle (3.2pt);
\draw [fill=black] (0.,9.514364454222578) circle (3.2pt);
\draw [fill=black] (-1.9562952014676087,9.09854107258706) circle (3.2pt);
\draw [fill=black] (-3.5743291902175023,7.922970568002118) circle (3.2pt);
\draw [fill=black] (-4.574329190217502,6.190919760433243) circle (3.2pt);
\draw [fill=black] (-4.78338611675281,4.201875969696695) circle (3.2pt);
\draw [fill=black] (-4.165352128002917,2.29976293710639) circle (3.2pt);
\draw [fill=black] (-2.8270909152852033,0.813473286151603) circle (3.2pt);
\draw [fill=black] (0.,11.) circle (3.2pt);
\draw [fill=black] (-5.987252556868822,6.650006391525967) circle (3.2pt);
\draw [fill=black] (5.987252556868827,6.6500063915259595) circle (3.2pt);
\draw [fill=black] (3.7003255793746446,-0.38843111782983275) circle (3.2pt);
\draw [fill=black] (-3.70032557937465,-0.3884311178298277) circle (3.2pt);
\end{tikzpicture}
\caption{Dual graph of the rational curves $R_1, \ldots, R_{20}$ in the Enriques surface of type VII. The picture combines \cite[Figure~8.16]{DK22} and \cite[Figure~8.17]{DK22}.}
\label{fig_dual_graph_K7}
\end{figure}

\newpage

\

%-------------------------------------------

%BIBLIOGRAPHY


\begin{thebibliography}{BHPV04}

\bibitem[BHPV04]{BHPV04}
Wolf~P. Barth, Klaus Hulek, Chris A.~M. Peters, and Antonius Van~de Ven.
\newblock {\em Compact complex surfaces}, volume~4 of {\em Ergebnisse der
  Mathematik und ihrer Grenzgebiete. 3. Folge. A Series of Modern Surveys in
  Mathematics}.
\newblock Springer-Verlag, Berlin, second edition, 2004.

\bibitem[BM01]{BM01}
Tom Bridgeland and Antony Maciocia.
\newblock Complex surfaces with equivalent derived categories.
\newblock {\em Math. Z.}, 236(4):677--697, 2001.

\bibitem[CD89]{CD89}
Fran\c{c}ois Cossec and Igor Dolgachev.
\newblock {\em Enriques surfaces. {I}}, volume~76 of {\em Progress in
  Mathematics}.
\newblock Birkh\"{a}user Boston, Inc., Boston, MA, 1989.

\bibitem[CDL22]{CDL22}
Fran\c{c}ois Cossec, Igor Dolgachev, and Christian Liedtke.
\newblock Enriques surfaces. {I}, 2022.
\newblock With an appendix by Shigeyuki Kond\=o. \url{http://www.math.lsa.umich.edu/~idolga/EnriquesOne.pdf} (version of June 24, 2022).

\bibitem[Cos83]{Cos83}
Fran\c{c}ois Cossec.
\newblock Reye congruences.
\newblock {\em Trans. Amer. Math. Soc.}, 280(2):737--751, 1983.

\bibitem[Cos85]{Cos85}
Fran\c{c}ois Cossec.
\newblock On the {P}icard group of {E}nriques surfaces.
\newblock {\em Math. Ann.}, 271(4):577--600, 1985.

\bibitem[DK22]{DK22}
Igor Dolgachev and Shigeyuki Kond\=o.
\newblock Enriques surfaces. {II}, 2022.
\newblock \url{http://www.math.lsa.umich.edu/~idolga/EnriquesTwo.pdf} (version of May 12, 2022).

\bibitem[DM19]{DM19}
Igor {Dolgachev} and Dimitri {Markushevich}.
\newblock {Lagrangian tens of planes, Enriques surfaces and holomorphic
  symplectic fourfolds}.
\newblock {\em arXiv e-prints}, page arXiv:1906.01445, June 2019.

\bibitem[Dol18]{Dol18}
Igor Dolgachev.
\newblock Salem numbers and {E}nriques surfaces.
\newblock {\em Exp. Math.}, 27(3):287--301, 2018.

\bibitem[FV21]{FV21}
Dino {Festi} and Davide~Cesare {Veniani}.
\newblock {Enriques involutions on pencils of K3 surfaces}.
\newblock {\em arXiv e-prints}, page arXiv:2103.07324, March 2021.

\bibitem[HLT21]{HLT21}
Katrina Honigs, Max Lieblich, and Sofia Tirabassi.
\newblock Fourier-{M}ukai partners of {E}nriques and bielliptic surfaces in
  positive characteristic.
\newblock {\em Math. Res. Lett.}, 28(1):65--91, 2021.

\bibitem[Kon86]{Kon86}
Shigeyuki Kond\=o.
\newblock Enriques surfaces with finite automorphism groups.
\newblock {\em Japan. J. Math. (N.S.)}, 12(2):191--282, 1986.

\bibitem[Kon12]{Kon12}
Shigeyuki Kond\=o.
\newblock The moduli space of {H}essian quartic surfaces and automorphic forms.
\newblock {\em J. Pure Appl. Algebra}, 216(10):2233--2240, 2012.

\bibitem[LNSZ21]{LNSZ21}
Chunyi Li, Howard Nuer, Paolo Stellari, and Xiaolei Zhao.
\newblock A refined derived {T}orelli theorem for {E}nriques surfaces.
\newblock {\em Math. Ann.}, 379(3-4):1475--1505, 2021.

\bibitem[LSZ22]{LSZ22}
Chunyi {Li}, Paolo {Stellari}, and Xiaolei {Zhao}.
\newblock {A Refined Derived Torelli Theorem for Enriques surfaces, II: the
  non-generic case}.
\newblock {\em Math. Z.}, 2022.
\newblock DOI: 10.1007/s00209-021-02930-4.

\bibitem[Mar19]{Mar19}
Gebhard Martin.
\newblock Enriques surfaces with finite automorphism group in positive
  characteristic.
\newblock {\em Algebr. Geom.}, 6(5):592--649, 2019.

\bibitem[MLP02]{MLP02}
Margarida Mendes~Lopes and Rita Pardini.
\newblock Enriques surfaces with eight nodes.
\newblock {\em Math. Z.}, 241(4):673--683, 2002.

\bibitem[MMV22a]{MMV22a}
Gebhard {Martin}, Giacomo {Mezzedimi}, and Davide~Cesare {Veniani}.
\newblock {On extra-special Enriques surfaces}.
\newblock {\em arXiv e-prints}, page arXiv:2201.05481, January 2022.

\bibitem[MMV22b]{MMV22b}
Gebhard {Martin}, Giacomo {Mezzedimi}, and Davide~Cesare {Veniani}.
\newblock {Enriques surfaces of non-degeneracy 3}.
\newblock {\em arXiv e-prints}, page arXiv:2203.08000, March 2022.

\bibitem[MRS22]{MRS22}
Riccardo Moschetti, Franco Rota, and Luca Schaffler.
\newblock SageMath code {C}nd{F}inder, 2022.
\newblock \url{https://github.com/rmoschetti/CNDFinder}.

\bibitem[{Nik}80]{Nik80}
Vyacheslav~V. {Nikulin}.
\newblock {Integral symmetric bilinear forms and some of their applications}.
\newblock {\em {Math. USSR, Izv.}}, 14:103--167, 1980.

\bibitem[{Oud}11]{Oud11}
R\'emy {Oudompheng}.
\newblock Periods of an arrangement of six lines and Campedelli surfaces.
\newblock {\em arXiv e-prints}, page arXiv:1106.4846, June 2011.

\bibitem[Par91]{Par91}
Rita Pardini.
\newblock Abelian covers of algebraic varieties.
\newblock {\em J. Reine Angew. Math.}, 417:191--213, 1991.

\bibitem[{Sag}22]{Sag22}
{The Sage Developers}.
\newblock {\em {S}ageMath, the {S}age {M}athematics {S}oftware {S}ystem
  ({V}ersion 9.4)}, 2022.
\newblock \url{https://www.sagemath.org}.

\bibitem[Sch18]{Sch18}
Luca Schaffler.
\newblock K3 surfaces with {$\mathbb{Z}_2^2$} symplectic action.
\newblock {\em Rocky Mountain J. Math.}, 48(7):2347--2383, 2018.

\bibitem[Sch22]{Sch22}
Luca Schaffler.
\newblock The {KSBA} compactification of the moduli space of
  {$D_{1,6}$}-polarized {E}nriques surfaces.
\newblock {\em Math. Z.}, 300(2):1819--1850, 2022.

\end{thebibliography}
\end{document}